\documentclass{amsart}

\usepackage{amsmath}
\usepackage{amssymb}
\usepackage{latexsym}
\usepackage{amsbsy}
\usepackage[all]{xy}
\usepackage{amscd}
\usepackage{mathrsfs}
\usepackage[dvips]{graphicx}

\newtheorem{thm}{Theorem}[section]
\newtheorem{prop}[thm]{Proposition}
\newtheorem{lem}[thm]{Lemma}
\newtheorem{cor}[thm]{Corollary}
\newtheorem{conj}[thm]{Conjecture}
\newtheorem{ques}[thm]{Question}

\theoremstyle{definition}

\theoremstyle{remark}
\newtheorem{rem}[thm]{Remark}

\numberwithin{equation}{section}

\DeclareMathOperator{\Ext}{Ext}
\DeclareMathOperator{\Hom}{Hom}

\DeclareMathOperator{\soc}{soc}
\DeclareMathOperator{\tp}{top}
\DeclareMathOperator{\dimv}{\underline{dim}}

\DeclareMathOperator{\Irr}{Irr}
\DeclareMathOperator{\wt}{wt}

\DeclareMathOperator{\GL}{GL}

\DeclareMathOperator{\id}{id}

\DeclareMathOperator{\conv}{conv}
\DeclareMathOperator{\Pol}{Pol}

\newcommand\bfi{\mathbf{i}}

\newcommand\tildef{\widetilde{f}}
\newcommand\tildee{\widetilde{e}}

\newcommand\hatfrakg{\hat{\mathfrak{g}}}

\newcommand\hatI{\hat{I}}
\newcommand\hatC{\widehat{C}}

\newcommand\hatf{\hat{f}}
\newcommand\hath{\hat{h}}
\newcommand\hatw{\hat{w}}
\newcommand\hatW{\widehat{W}}
\newcommand\hats{\hat{s}}
\newcommand\hatM{\widehat{M}}
\newcommand\hatP{\widehat{P}}
\newcommand\hatmu{\hat{\mu}}
\newcommand\hatgamma{\hat{\gamma}}
\newcommand\hatpi{\hat{\varpi}}
\newcommand\hatMV{\widehat{\mathcal{MV}}}

\begin{document}

\title[The crystal structure for MV polytopes]{An insight into the description of the crystal structure for Mirkovi\'{c}-Vilonen polytopes}

\author{Yong Jiang}

\address{Fakult\"{a}t f\"{u}r Mathematik, Universit\"{a}t Bielefeld, Postfach 10 01 31, D-33501 Bielefeld, Germany}

\curraddr{Leibniz-Institut f\"{u}r Pflanzengenetik und Kulturpflanzenforschung (IPK), Corrensstrasse 3, D-06466 Stadt Seeland OT Gatersleben, Germany}

\email{jiang@ipk-gatersleben.de}

\thanks{Y. Jiang was supported by the Sonderforschungsbereich 701 in Universit\"{a}t Bielefeld.}

\author{Jie Sheng}

\address{Department of Applied Mathematics, China Agricultural University, 100083 Beijing, P.R.China}

\address{Mathematisches Institut, Universit\"{a}t Bonn, Endenicher Allee 60, 53115 Bonn, Germany}

\email{shengjie@amss.ac.cn}

\thanks{J. Sheng was supported by NSF of China (No.11301533).}





\begin{abstract}
  We study the description of the crystal structure on the set of Mirkovi\'{c}-Vilonen polytopes. Anderson and Mirkovi\'{c} defined an operator and conjectured that it coincides with the Kashiwara operator. Kamnitzer proved the conjecture for type $A$ and gave an counterexample for type $C_{3}$. He also gave an explicit formula to calculate the Kashiwara operator for type $A$. In this paper we prove that a part of the AM conjecture still holds in general, answering an open question of Kamnitzer (2007). Moreover, we show that although the formula given by Kamnitzer does not hold in general, it is still valid in many cases regardless of the type. The main tool is the connection between MV polytopes and preprojective algebras developed by Baumann and Kamnitzer.
\end{abstract}

\maketitle

\section{Introduction}\label{sec introduction}
Let $G$ be a complex connected reductive group and $G^{\vee}$ be its Langlands dual group. In the geometric Satake correspondence, the intersection cohomology of the affine Grassmanian associated with $G$ provides a geometric realization of the irreducible highest weight representations $V(\lambda)$ of $G^{\vee}$ and a family of subvarieties, called Mirkovi\'{c}-Vilonen (MV) cycles, forms a basis of $V(\lambda)$ \cite{MV}. Attempting to understand combinatorial aspects of $V(\lambda)$ using MV cycles, Anderson defined MV polytopes as moment polytopes of MV cycles \cite{An}. He proved that MV polytopes can be used to count weight multiplicities as well as tensor product multiplicities without giving a full combinatorial description of the polytopes. Later, Kamnitzer gave a complete combinatorial characterization of MV polytopes as pseudo-Weyl polytopes satisfying tropical Pl\"{u}cker relations \cite{Kam2} (see Section \ref{subsec MV polytopes}).

The set of MV polytopes, denoted by $\mathcal{MV}$, naturally inherits a crystal structure via an explicit bijection to Lusztig's canonical basis \cite{Kam2} (see Theorem \ref{thm Kamnitzer 1}). Let $\tildef_{j}$ be the Kashiwara operator. Kamnizter studied the crystal structure on $\mathcal{MV}$ and gave a description of $\tildef_{j}(P)$ for polytope $P\in\mathcal{MV}$ \cite{Kam1} (see Theorem \ref{thm Kamnitzer}). However, this description is non-explicit because it requires to solve many equations given by the tropical Pl\"{u}cker relations, which involves addition and taking minimum.

Anderson and Mirkovi\'{c} proposed a conjecture to describe the crystal structure on $\mathcal{MV}$, which was called Anderson-Mirkovi\'{c} (AM) conjecture in \cite{Kam1} (see Conjecture \ref{conj AM}). More precisely, they defined a new operator $AM_{j}$ acting on $\mathcal{MV}$ and conjectured that it is the same as the Kashiwara operator $\tildef_{j}$. By definition, for any MV polytope $P$, $AM_{j}(P)$ is the smallest pseudo-Weyl polytope satisfying four conditions concerning $P$ (see Section \ref{subsec the AM conjecture}). The description given by the operator $AM_{j}$ does not require the tropical Pl\"{u}cker relations and hence was considered more explicit than the one given by Kamnitzer.

In case of type $A$, Kamnitzer proved the AM conjecture and gave an explicit formula for the Kashiwara operators \cite{Kam1}. He also gave a counterexample of the conjecture for type $C_{3}$. An alternative proof for type $A$ was given by Saito \cite{Sai2} via the connection between MV polytopes and representations of quivers. Naito and Sagaki \cite{NS} proved modified versions of the AM conjecture for type $B$, $C$ and also gave explicit formulas for the Kashiwara operators using diagram automorphisms. As far as we know, up to now there is no result concerning type $D$, $E$ or other non-simply-laced types.

Our study was motivated by an interesting question raised by Kamnitzer in \cite{Kam1}. He found in his counterexample for type $C_{3}$ that $AM_{j}(P)\subset\tildef_{j}(P)$. It means that $\tildef_{j}(P)$ is indeed a pseudo-Weyl polytope satisfying the four conditions required by $AM_{j}$, but it is not the smallest one. Hence he asked:
\begin{ques}[\cite{Kam1}, Question 2]\label{ques Kamnitzer}
Is $AM_{j}(P)$ always contained in $\tildef_{j}(P)$?
\end{ques}

The aim of this paper is to give an affirmative answer to the above question. Namely, we proved that for any MV polytope $P$, $AM_{j}(P)\subseteq\tildef_{j}(P)$ holds in general (see Theorem \ref{thm main simply-laced case}). In our proof we treat the simply-laced case and the non-simply-laced case separately. In the simply-laced case, Baumann and Kamnitzer has recently established a link between MV polytopes and representations of preprojective algebras \cite{BK}. Their results shed new light on the theory of MV polytopes and is the main tool for us to prove our theorem. In the non-simply-laced case, a diagram automorphism $\sigma$ on the Lie algebra $\mathfrak{g}$ induces a action on $\mathcal{MV}$ and the set $\hatMV$ for $\mathfrak{g}^{\sigma}$ is in bijection with the $\sigma$-invariant MV polytopes in $\mathcal{MV}$ (\cite{NS}, \cite{H}). Therefore, having the theorem proved for the simply-laced case, we can use diagram automorphisms to obtain the corresponding results for the non-simply-laced case.

We also show that although the explicit formula given by Kamnitzer in \cite{Kam1} for type $A$ does not hold in general, it is still true in many special cases (see Theorem \ref{thm minuscule} and \ref{thm simply-laced case}). All our results are type-independent.

Note that it remains open whether the AM conjecture holds for all simply-laced cases as we did not find any counterexamples in type $D$ or $E$.

The paper is organized as follows. In Section \ref{sec MV polytopes and crystal}, we recall the definition of MV polytopes and the crystal structure on the set of them. In Section \ref{sec AM conjecture and main results}, we introduce the AM conjecture and state our main results. We then recall the link between MV polytopes and representations of preprojective algebras in Section \ref{sec MV polytopes and preproj algs}. The proofs of our results are presented in the last two sections. The simply-laced case is treated in Section \ref{sec proof simply-laced}, while in Section \ref{sec proof non-simply-laced} we focus on the non-simply-laced case.

\section{Mirkovi\'{c}-Vilonen polytopes and the crystal structure}\label{sec MV polytopes and crystal}

\subsection{Notations}\label{subsec notations}
Let $\mathfrak{g}$ be a finite-dimensional complex simple Lie algebra of rank $n$ and $\mathfrak{h}$ be a Cartan subalgebra of $\mathfrak{g}$. Let $C=(c_{ij})$ be the Cartan matrix. The vertices of the Dynkin diagram of $\mathfrak{g}$ is denoted by $I=\{1,\ldots,n\}$. Let $Q$ be the root system and $Q_{+}$ (resp. $Q_{-}$) be the positive (resp. negative) root lattice. For $1\leq i \leq n$, denote by $\alpha_{i}$ (resp. $h_{i}$) the simple root (resp. simple coroot). Let $P$ be the weight lattice and $P_{+}$ be the set of dominant weights. The $i$-th fundamental weight is denoted by $\varpi_{i}$. Let $P^{\ast}$ be the coweight lattice and $\mathfrak{h}_{\mathbb{R}}=P^{\ast}\otimes\mathbb{R}$ be the real form of $\mathfrak{h}$.

Let $W$ be the Weyl group of $\mathfrak{g}$. For any $w\in W$, the length of $w$ is denoted by $\ell(w)$. Let $w_{0}$ be the longest element in $W$ and $r=\ell(w_{0})$. Denote by $s_{1},\ldots,s_{n}$ the simple reflections. A Weyl group translate of a fundamental weight $w\varpi_{i}$ is called a chamber weight. The collection of all chamber weights is denoted by $\Gamma=\{w\varpi_{i}|w\in W,i\in I\}$.

The Bruhat order on $W$ is denoted by $\geq$. We will use the same notation for the usual partial order on $P^{\ast}$, namely $\mu\geq\nu$ if and only if $\mu-\nu\in\sum_{i\in I}\mathbb{N}h_{i}$. The twisted partial order $\geq_{w}$ for $w\in W$ is defined as $\mu\geq_{w}\nu$ if and only if $w^{-1}\mu\geq w^{-1}\nu$.

Let $\mathfrak{g}^{\vee}$ be the Langlands dual of $\mathfrak{g}$, namely the Lie algebra whose root system is dual to that of $\mathfrak{g}$. In particular, the weight lattice of $\mathfrak{g}^{\vee}$ is $P^{\ast}$.

\subsection{Mirkovi\'{c}-Vilonen polytopes}\label{subsec MV polytopes}
We recall the combinatorial definition of MV polytopes following \cite{Kam1}.

For a collection of integers $M_{\bullet}=(M_{\gamma})_{\gamma\in\Gamma}$ indexed by the set of chamber weights, the following inequalities are called \textit{edge inequalities}:
\begin{equation}\label{equ edge inequality}
M_{w\varpi_{i}}+M_{ws_{i}\varpi_{i}}+\sum_{j\neq i}c_{ji}M_{w\varpi_{j}}\leq 0,\quad \text{for all }i\in I,w\in W.
\end{equation}

Given $M_{\bullet}$ satisfying the edge inequalities, we have the associated \textit{pseudo-Weyl polytope}
$$P(M_{\bullet})=\{h\in\mathfrak{h}_{\mathbb{R}}|\langle h,\gamma\rangle\geq M_{\gamma}\ \text{for all}\ \gamma\}.$$

As shown in \cite{Kam2}, there is a map $w\mapsto \mu_{w}$ from the Weyl group onto the set of vertices of the polytope $P(M_{\bullet})$ such that
$$\langle\mu_{w},w\varpi_{i}\rangle=M_{w\varpi_{i}}, \quad \text{for all } i \in I,w\in W.$$
The collection $\mu_{\bullet}=(\mu_{w})_{w\in W}$ is called the \textit{Gelfand-Goresky-MacPherson-Serganova (GGMS) datum} of the pseudo-Weyl polytope $P(M_{\bullet})$. And we have
$$P(M_{\bullet})=\conv(\mu_{\bullet})=\{h\in\mathfrak{h}_{\mathbb{R}}|h\geq_{w}\mu_{w},\ \text{for all}\ w\in W\}.$$

Let $w\in W$ and $i,j\in I$ be such that $ws_{i}>w$, $ws_{j}>w$ and $i\neq j$. We say that $M_{\bullet}=(M_{\gamma})_{\gamma\in\Gamma}$ satisfies the \textit{tropical Pl\"{u}cker relation} at $(w,i,j)$ if $c_{ij}=0$ or $c_{ij}=c_{ji}=-1$ and
\begin{equation}\label{equ tropical Plucker relation}
M_{ws_{i}\varpi_{i}}+M_{ws_{j}\varpi_{j}}=\min\{M_{w\varpi_{i}}+M_{ws_{i}s_{j}\varpi_{j}},M_{ws_{j}s_{i}\varpi_{i}}+M_{w\varpi_{j}}\}.
\end{equation}
The relations for other possible values of $c_{ij},c_{ji}$ are omitted as they will not be used in this paper. We refer to \cite{Kam2} for details.

We say that $M_{\bullet}$ satisfies the tropical Pl\"{u}cker relations if it satisfies the tropical Pl\"{u}cker relation at each $(w,i,j)$.

The pseudo-Weyl polytope $P(M_{\bullet})$ is called a \textit{Mirkovi\'{c}-Vilonen (MV) polytope} of weight $(\mu_{1},\mu_{2})$ if $M_{\bullet}$ satisfies the edge inequalities and $\mu_{e}=\mu_{1}$ and $\mu_{w_{0}}=\mu_{2}$. In this case, $M_{\bullet}$ is called a \textit{Berenstein-Zelevinsky (BZ) datum} of weight $(\mu_{1},\mu_{2})$.

The weight lattice $P$ acts on $\mathfrak{h}_{\mathbb{R}}$ by translation. Hence we have an action of $P$ on the set of MV polytopes. Namely for $\nu\in P$, $\nu+P(M_{\bullet})=P(M_{\bullet}')$ where $M_{\gamma}'=M_{\gamma}+\langle\nu,\gamma\rangle$. The orbit of an MV polytope of weight $(\mu_{1},\mu_{2})$ is called a \textit{stable MV polytope} of weight $(\mu_{1}-\mu_{2})$. Let $\mathcal{MV}$ denote the set of stable MV polytopes. In this paper, we always consider representatives in $\mathcal{MV}$ with $\mu_{w_{0}}=0$.

\subsection{The crystal structure}\label{subsec the crystal structure}
We recall the definition of abstract crystals following \cite{Ka2}. For more details on the theory of crystal bases we refer to \cite{Ka1}.

A $\mathfrak{g}$-\textit{crystal} is a set $B$ with maps $\wt:B\rightarrow P$, $\varepsilon_{i},\varphi_{i}:B\rightarrow\mathbb{Z}\cup\{-\infty\}$ and $\tildee_{i},\tildef_{i}:B\rightarrow B\cup\{0\}$ for all $i\in I$, satisfying the following axioms:

(1). For any $b\in B$, $\varphi_{i}(b)=\varepsilon_{i}(b)+\langle h_{i},\wt(b)\rangle$.

(2). If $\tildee_{i}(b)\in B$, then $\wt(\tildee_{i}(b))=\wt(b)+\alpha_{i}$, $\varepsilon_{i}(\tildee_{i}(b))=\varepsilon_{i}(b)-1$ and $\varphi_{i}(\tildee_{i}(b))=\varphi_{i}(b)+1$.

(3). If $\tildef_{i}(b)\in B$, then $\wt(\tildef_{i}(b))=\wt(b)-\alpha_{i}$, $\varepsilon_{i}(\tildef_{i}(b))=\varepsilon_{i}(b)+1$ and $\varphi_{i}(\tildef_{i}(b))=\varphi_{i}(b)-1$.

(4). For $b,b'\in B$, we have $b'=\tildee_{i}(b)$ if and only if $b=\tildef_{i}(b')$.

(5). For $b\in B$, if $\varphi_{i}(b)=-\infty$, then $\tildee_{i}(b)=\tildef_{i}(b)=0$.

Let $\mathbf{B}$ be Lusztig's canonical basis for the positive part of the quantized enveloping algebra $U_{q}^{+}(\mathfrak{g})$ \cite{Lu1}. It is known that $\mathbf{B}$ carries a $\mathfrak{g}$-crystal structure.

Let $\bfi=(i_{1},\ldots,i_{r})$ be a reduced word of $w_{0}$. For $1\leq k\leq r$, let $w_{k}^{\bfi}=s_{i_{1}}\cdots s_{i_{k}}$. The \textit{$\bfi$-Lusztig datum} of an MV polytope $P(M_{\bullet})$ is $n_{\bullet}=(n_{1},\ldots,n_{r})\in\mathbb{N}^{r}$ defined by
$$n_{k}=-M_{w_{k-1}^{\bfi}\varpi_{i_{k}}}-M_{w_{k}^{\bfi}\varpi_{i_{k}}}-\sum_{j\neq i_{k}}a_{ji_{k}}M_{w_{k}^{\bfi}\varpi_{j}}.$$

The meaning of the $\bfi$-Lusztig datum of an MV polytope is the following \cite{Kam1}: the reduced word $\bfi$ determines a path $e=\mu_{w_{0}^{\bfi}}$, $\mu_{w_{1}^{\bfi}}$, $\cdots$, $\mu_{w_{r}^{\bfi}}=w_{0}$ through the $1$-skeleton of the polytope. The integers $(n_{1},\ldots,n_{r})$ are just the lengths of the edges along the path.

Let $\mathbf{B}^{\vee}$ be the canonical basis of $U_{q}^{+}(\mathfrak{g}^{\vee})$. For any reduced word $\bfi$ of $w_{0}$, Lusztig showed that there is a bijection $\phi_{\bfi}:\mathbf{B}^{\vee}\rightarrow\mathbb{N}^{r}$, known as Lusztig's parametrization of the canonical basis $\mathbf{B}^{\vee}$. $\phi_{\bfi}(b)$ is called the $\bfi$-Lusztig datum of $b$.

\begin{thm}[\cite{Kam2}]\label{thm Kamnitzer 1}
There is a coweight-preserving bijection $b\mapsto P(b)$ between the canonical basis $\mathbf{B}^{\vee}$ and the set of MV polytopes $\mathcal{MV}$. Under this bijection, the $\bfi$-Lusztig datum of $b$ equals that of $P(b)$.
\end{thm}

Thus the set $\mathcal{MV}$ inherits a crystal structure from the canonical basis. The action of the Kashiwara operator $\tildef_{j}$ ($j\in I$) can be described as follows:

\begin{thm}[\cite{Kam1}]\label{thm Kamnitzer}
Let $P=P(M_{\bullet})$ be an MV polytope and assume that $\widetilde{f}_{j}P=P(M_{\bullet}')$. Then $M_{\bullet}'$ is uniquely determined by $M_{\varpi_{j}}'=M_{\varpi_{j}}-1$ and $M_{\gamma}'=M_{\gamma}$ for $\gamma=w\varpi_{i}$ with $s_{j}w<w$.
\end{thm}

As pointed out in \cite{Kam1}, this description is non-explicit because the rest of $M_{\gamma}'$ are determined by the tropical Pl\"{u}cker relations. To calculate them one needs to recursively solve a number of equations involving $+$ and taking minimum. Hence it is interesting to seek a more explicit description or even an explicit formula to calculate the rest of $M_{\gamma}'$.

\section{The Anderson-Mirkovi\'{c} conjecture and main results}\label{sec AM conjecture and main results}

\subsection{The Anderson-Mirkovi\'{c} conjecture}\label{subsec the AM conjecture}
For $j\in I$ we define $W_{j}^{-}=\{w\in W|s_{j}w<w\}$, $W_{j}^{+}=\{w\in W|s_{j}w>w\}$. Let $\Gamma^{j}=\cup_{i\in I}W_{j}^{-}\varpi_{i}$ and $\Gamma_{j}=\Gamma\setminus\Gamma^{j}$.

Let $P=P(M_{\bullet})$ be an MV polytope with GGMS datum $\mu_{\bullet}$ and fix $j\in I$. Let $c_{j}(P)=M_{\varpi_{j}}-M_{s_{j}\varpi_{j}}-1$. We define a map  $r_{j}:\mathfrak{h}_{\mathbb{R}}\rightarrow\mathfrak{h}_{\mathbb{R}}$ by $r_{j}(\alpha)=s_{j}(\alpha)+c_{j}h_{j}$. Note that $r_{j}(\mu_{s_{j}})=\mu_{e}-h_{j}$. The $j$-th \textit{AM operator} $AM_{j}:\mathcal{MV}\rightarrow\mathcal{MV}$ is defined as following:
$P'=AM_{j}(P)=\conv(\mu_{\bullet}')$ is the smallest pseudo-Weyl polytope such that

(i). $\mu'_{w}=\mu_{w}$ for all $w\in W_{j}^{-}$,

(ii). $\mu'_{e}=\mu_{e}-h_{j}$,

(iii). $P'$ contains $\mu_{w}$ for all $w\in W_{j}^{+}$,

(iv). if $w\in W_{j}^{-}$ is such that $\langle\mu_{w},\alpha_{j}\rangle\geq c_{j}(P)$, then $P'$ contains $r_{j}(\mu_{w})$.

\begin{conj}[Anderson-Mirkovi\'{c}]\label{conj AM}
For any MV polytope $P$ and any $j\in I$, $AM_{j}(P)$ is an MV polytope. Moreover $AM_{j}(P)=\tilde{f}_{j}(P)$.
\end{conj}

If the conjecture holds, we have a more explicit description of the action of Kashiwara operators in the sense that the AM operator only requires the smallest pseudo-Weyl polytope without involving tropical Pl\"{u}cker relations.

Kamnitzer has proved the following result:
\begin{prop}[\cite{Kam1}, Proposition 5.3]\label{prop Kamnitzer}
Let $P=P(M_{\bullet})$ be an MV polytope and let $M''_{\bullet}$ be defined as
\begin{equation}\label{equ M'}
M''_{\gamma}=\begin{cases}
            & M_{\gamma},\ \ \ \text{if } \gamma\in\Gamma^{j}\\
            & \min\{M_{\gamma}, M_{s_{j}\gamma}+c_{j}(P)\langle h_{j},\gamma\rangle\}, \ \ \ \text{if } \gamma\in\Gamma_{j}.
            \end{cases}
\end{equation}
If $M''_{\bullet}$ satisfy the edge inequalities, then $P(M''_{\bullet})=AM_{j}(P)$.
\end{prop}

\begin{rem}\label{rem formula implies AM conj}
Let $\tilde{f}_{j}(P)=P(M'_{\bullet})$. Suppose that $M'_{\bullet}$ is defined as (\ref{equ M'}). Then $M'_{\bullet}$ automatically satisfies the edge inequalities since $P(M'_{\bullet})$ is already an MV poplytope. Thus by the above proposition, Conjecture \ref{conj AM} holds.
\end{rem}

In the case of type $A$, Kamnitzer proved Conjecture \ref{conj AM} by showing that (\ref{equ M'}) indeed holds for $\tilde{f}_{j}(P)=P(M'_{\bullet})$ \cite{Kam1}. He also gave a counterexample of the conjecture for type $C_{3}$. Naito and Sagaki \cite{NS} proved modified versions of Conjecture \ref{conj AM} for type $B$ and $C$. So far we have not seen any results concerning type $D$, $E$ or other non-simply-laced types.

\subsection{Main results}\label{subsec main results}
In this subsection we state our main results. The proofs will be given in the next two sections.

Our first result answers an open question raised by Kamnitzer in 2007 (Question \ref{ques Kamnitzer}).

\begin{thm}\label{thm main simply-laced case}
For any MV polytope $P$ and any $j\in I$, we have $AM_{j}(P)\subseteq\tildef_{j}(P)$.
\end{thm}

That is, the MV polytope $\tildef_{j}(P)$ satisfies the condition (i)-(iv) in the definition of AM operators. For type $D$ and $E$, it remains open whether it is the smallest pseudo-Weyl polytope satisfying those conditions.

The next result shows that for particular $\gamma\in\Gamma_{j}$, the BZ datum of $\tildef_{j}(P)$ has the form of (\ref{equ M'}).

\begin{thm}\label{thm minuscule}
Let $P=P(M_{\bullet})$ be an MV polytope and $j\in I$. Let $\tildef_{j}(P)=P(M'_{\bullet})$. If $\langle h_{j},\gamma\rangle=1$, then $M'_{\gamma}=\min\{M_{\gamma}, M_{s_{j}\gamma}+c_{j}(P)\}$.
\end{thm}

\begin{rem}\label{rem thm2 implies AM type A}
Note that in the case of type $A$, all $\gamma\in\Gamma_{j}$ satisfies the condition $\langle h_{j},\gamma\rangle=1$. Hence by proving this theorem we obtain a new proof of Conjecture \ref{conj AM} for type $A$.
\end{rem}

Let $P=P(M_{\bullet})$ be an MV polytope and $\tildef_{j}(P)=P(M'_{\bullet})$. In any simply-laced case other than type $A$,
$M'_{\gamma}$ might not satisfy (\ref{equ M'}), for $\gamma\in\Gamma_{j}$ such that $\langle h_{j},\gamma\rangle>1$. We will give an example for type $D_{4}$ in Section \ref{subsec eg of type D}, where $M'_{\bullet}$ is indeed not given by (\ref{equ M'}) but the AM conjecture still holds.

\begin{thm}\label{thm simply-laced case}
Let $P=P(M_{\bullet})$ be an MV polytope and $j\in I$. Set
$$m=m(j,P)=M_{\varpi_{j}}-M_{-s_{j}\varpi_{j}}-M_{s_{j}\varpi_{j}}.$$
Then $m\geq 0$ and for any positive integer $k>m$, $AM_{j}(\tildef_{j}^{k}(P))=\tildef_{j}^{k+1}(P)$.

Moreover, let $P^{(k)}=P(M^{(k)}_{\bullet})=\tildef_{j}^{k}(P)$ and $P^{(k+1)}=P(M^{k+1}_{\bullet})=\tildef_{j}^{k+1}(P)$, we have
\begin{equation}\label{equ more explicit M'}
M^{(k+1)}_{\gamma}=\begin{cases}
            & M^{(k)}_{\gamma},\ \ \ \text{if } \gamma\in\Gamma^{j}\\
            & M^{(k)}_{s_{j}\gamma}+c_{j}(P^{(k)})\langle h_{j},\gamma\rangle, \ \ \ \text{if } \gamma\in\Gamma_{j}.
            \end{cases}
\end{equation}
\end{thm}

Recall that for any MV polytope $P$, the set $\{\tildef_{j}^{k}(P)|k\in\mathbb{Z}\}$ is called a \textit{$j$-string} in the crystal graph (see \cite{Ka2}). In any $j$-string, there exists a unique polytope $P_{0}$ such that $\varepsilon_{j}(P_{0})=0$. We call $P_{0}$ a \textit{$j$-extremal} polytope. The above theorem indicates that for each $j$-string, $AM_{j}=\tildef_{j}$ holds for all MV polytopes $P$ lying far enough from the $j$-extremal polytope and the BZ-datum of $\tildef_{j}(P)$ can be explicitly calculated. So for each $j$-string, there are infinitely many MV polytopes for which $AM_{j}=\tildef_{j}$ holds.

\subsection{A characterization of AM operators}\label{subsec characterize AM operators}
The original definition of AM operators uses the GGMS datum $\mu_{\bullet}$. In this subsection we give a characterization of AM operators using the BZ datum $M_{\bullet}$, which will be more convenient for us to prove our results.

\begin{lem}\label{lem Gamma_j Gamma^j}
For $\gamma\in\Gamma$, we have $\gamma\in\Gamma^{j}\Longleftrightarrow\langle h_{j},\gamma\rangle\leq 0$, $\gamma\in\Gamma_{j}\Longleftrightarrow\langle h_{j},\gamma\rangle>0$.
\end{lem}
\begin{proof}
If $\gamma\in\Gamma^{j}$, then by definition there exist $w\in W_{j}^{-}$ and $i\in I$ such that $\gamma=w\varpi_{i}$. Thus $\langle h_{j},\gamma\rangle=\langle w^{-1}(h_{j}),\varpi_{i}\rangle\leq 0$. Since $s_{j}w<w$, we know that $w^{-1}(h_{j})\in Q_{-}$. Hence $\langle h_{j},\gamma\rangle\leq 0$. Similarly, if $\gamma\in\Gamma_{j}$, we have $\langle h_{j},\gamma\rangle\geq 0$.

Now we only need to prove that $\langle h_{j},\gamma\rangle=0$ implies $\gamma\in\Gamma_{j}$. Assume that $\gamma=w\varpi_{i}$ for some $w\in W$ and $i\in I$. If $w\in W_{j}^{-}$ we are done. If not, let $w'=s_{j}w$ then $\gamma=w'\varpi_{i}$ and $w'\in W_{j}^{-}$.
\end{proof}

The following lemma is in fact a generalization of \cite{Kam1} Proposition 5.3 and the proof is similar. Nevertheless, for readers' convenience, we present a complete proof here.

\begin{lem}\label{lem AM op in M}
Let $P=P(M_{\bullet})$ be an MV polytope, $j\in I$ and $P'=P(M'_{\bullet})$ be the smallest pseudo-Weyl polytope satisfying the following three conditions:

(a). $M'_{\gamma}=M_{\gamma}$, for all $\gamma\in\Gamma^{j}$;

(b). $M'_{\varpi_{i}}=M_{\varpi_{i}}$ for $i\neq j$ and $M'_{\varpi_{j}}=M_{\varpi_{j}}-1$;

(c). $M'_{\gamma} \leq \min\{M_{\gamma}, M_{s_{j}\gamma}+c_{j}\langle h_{j},\gamma\rangle\}$, for all $\gamma\in\Gamma_{j}$.

Then $P'=AM_{j}(P)$.
\end{lem}

\begin{proof}
First we need to prove that $P'$ satisfies the conditions (i) to (iv) in the definition of AM operators.

It is easy to see that (a) and (b) are equivalent to (i) and (ii) respectively. By (c) we know that $M'_{\gamma}\leq M_{\gamma}$ for all $\gamma\in\Gamma_{j}$. Together with (a) we have $M'_{\gamma}\leq M_{\gamma}$ for all $\gamma\in\Gamma$, which by definition implies (iii).

Now it remains to prove (iv). Let $w\in W_{j}^{-}$ such that $\langle\mu_{w}, \alpha_{j}\rangle\geq c_{j}$. We need to show that $r(\mu_{w})\in P'$, i.e. $\langle r(\mu_{w}),\gamma\rangle\geq M_{\gamma}'$ for all $\gamma\in\Gamma$.

First we consider $\gamma=s_{j}v\varpi_{i}$ for any $v\in W_{j}^{-}$ and $i\in I$. By definition of MV polytopes, we have $\langle\mu_{w},v\varpi_{i}\rangle\geq M_{v\varpi_{i}}=\langle\mu_{v},v\varpi_{i}\rangle$. From this inequality we can deduce that $\langle r_{j}(\mu_{w}),\gamma\rangle\geq\langle r_{j}(\mu_{v}),\gamma\rangle$. We claim that $\langle r_{j}(\mu_{v}),\gamma\rangle\geq M'_{\gamma}$. Hence $\langle r_{j}(\mu_{w}),\gamma\rangle\geq M'_{\gamma}$.

To prove the claim, note that
$$\langle r(\mu_{w}),\gamma\rangle=\langle s_{j}(\mu_{w})+c_{j}h_{j}, s_{j}w\varpi_{i}\rangle=M_{w\varpi_{i}}+c_{j}\langle h_{j},\gamma\rangle.$$
It is clear that $\gamma\in\Gamma_{j}$ as $v\in W_{j}^{-}$. By our assumption (c) we have $M'_{\gamma}\leq M_{w\varpi_{i}}+c_{i}\langle h_{i},\gamma\rangle$.

Next we consider $\gamma=v\varpi_{i}$ for any $v\in W_{j}^{-}$ and $i \in I$. Since $v\in W_{j}^{-}$, we know that $\langle h_{j},\gamma\rangle\leq 0$. Let $d_{j}=\langle\mu_{w}, \alpha_{j}\rangle-c_{j}$. Then $d_{j}\ge 0$ and $r(\mu_{w})=\mu_{w}-d_{j}h_{j}$. We have
$$\langle r(\mu_{w}),\gamma\rangle=\langle\mu_{w},\gamma\rangle-d_{j}\langle h_{j},\gamma\rangle\geq\langle\mu_{w},\gamma\rangle.$$
Note that $\langle\mu_{w},\gamma\rangle\geq M_{\gamma}$. Also we know that $M_{\gamma}=M_{\gamma}'$ since $\gamma=v\varpi_{i}\in\Gamma^{j}$. Hence $\langle r(\mu_{w}),\gamma\rangle\geq M_{\gamma}'$.

Since $W_{j}^{-}\cup s_{j}W_{j}^{-}=W$ we see that $\langle r(\mu_{w}),\gamma\rangle\geq M_{\gamma}'$ for all $\gamma\in \Gamma$. We have proved that condition (iv) is satisfied.

Now, let $P''=P(M''_{\bullet})$ be a pseudo-Weyl polytope satisfying (i) to (iv), we need to prove $P'\subseteq P''$, for which we only need to show that $P''$ satisfies (a), (b) and (c).

We only need to prove (c). Since $P''$ satisfies (iii), we have that $M''_{\gamma}\leq M_{\gamma}$ for all $\gamma\in\Gamma_{j}$. It remains to prove $M''_{\gamma}\leq M_{s_{j}\gamma}+c_{j}\langle h_{j},\gamma \rangle$ for all $\gamma\in\Gamma_{j}$.

There exists $w\in W_{j}^{-}$ such that $\gamma=s_{j}w\varpi_{i}$ since $\gamma\in\Gamma_{j}$. If $\langle \mu_{w}, \alpha_{j}\rangle\geq c_{j}$, then $r_{j}(\mu_{w})\in P''$ by (iv). We have
$$M''_{\gamma}\leq\langle r_{j}(\mu_{w}),\gamma\rangle=M_{s_{j}\gamma}+c_{j}\langle h_{j},\gamma\rangle.$$

If $\langle \mu_{w}, \alpha_{j}\rangle < c_{j}$, we have
$$M''_{\gamma}<M_{\gamma}\leq\langle \mu_{w},\gamma\rangle<\langle r_{j}(\mu_{w}),\gamma\rangle=M_{s_{j}\gamma}+c_{j}\langle h_{j},\gamma\rangle$$
as desired.
\end{proof}

\begin{rem}\label{rem enough to prove main}
Assume that $P=P(M_{\bullet})$ is an MV polytope and $\tildef_{j}(P)=P(M'_{\bullet})$. By Theorem \ref{thm Kamnitzer} we know that $M'_{\bullet}$ satisfies the condition (a) and (b) in the above lemma. Therefore, to prove Theorem \ref{thm main simply-laced case}, it is enough to show that $M'_{\bullet}$ satisfies the condition (c).
\end{rem}

\section{MV polytopes and preprojective algebras}\label{sec MV polytopes and preproj algs}
In this section we briefly recall the relationship between MV polytopes and representations of preprojective algebras developed by Baumann and Kamnitzer \cite{BK}. It is the main tool to prove our results in the simply-laced case. We will also need some knowledge on the geometric realization of crystals via nilpotent varieties \cite{KS}, which are certain varieties of modules over preprojective algebras.

\subsection{Preprojective algebras and their modules}\label{subsec Preproj algs and modules}
Let $Q$ be a quiver whose underlying graph is the Dynkin diagram of $\mathfrak{g}$. Let $\Omega$ be the set of arrows in $Q$ and $s,t:\Omega\rightarrow I$ be the two maps indicating the source and target of an arrow. We define a new quiver $\overline{Q}$ from $Q$ by adding to each arrow $a\in\Omega$ an opposite arrow $a^{\ast}$. Set $\Omega^{\ast}=\{a^{\ast}|a\in\Omega\}$. So the set of arrows in $\overline{Q}$ is $H=\Omega\cup\Omega^{\ast}$. For each $a\in\Omega$, set $(a^{\ast})^{\ast}=a$. We also define a map $\epsilon:H\rightarrow\{\pm 1\}$ by assigning $\epsilon(a)=1$ if $a\in\Omega$ and $\epsilon(a)=-1$ if $a\in\Omega^{\ast}$. The \textit{preprojective algebra} $\Lambda=\Lambda(Q)$ is defined as the path algebra $\mathbb{C}\overline{Q}$ modulo the two-sided ideal generated by $\sum_{a\in H}\epsilon(a)aa^{\ast}$ (see \cite{R} for details).

A $\Lambda$-module $X$ consists of an $I$-graded $\mathbb{C}$-vector space $X=\oplus_{i\in I}X_{i}$ and a collection of linear maps $X_{a}:X_{s(a)}\rightarrow X_{t(a)}$ for each arrow $a\in H$ satisfying the following relations:
$$\sum_{a\in H, t(a)=i}\epsilon(a)X_{a}X_{a^{\ast}}=0,\ \text{for each}\ i\in I.$$

In our case ($\mathfrak{g}$ is a simple Lie algebra) it is known that $\Lambda$ is a finite-dimensional algebra. Denote by $\bmod\Lambda$ the category of finite-dimensional $\Lambda$-modules. For any $\Lambda$-module $X$, we call $\dimv X=(\dim X_{i})_{i\in I}\in\mathbb{N}^{n}$ the dimension vector of $X$. It will also be convenient for us to treat $\dimv X=\sum_{i\in I}(\dim X_{i})\alpha_{i}$ as an element in the positive root lattice. Let $(-,-):\mathfrak{h}^{\ast}\times\mathfrak{h}^{\ast}\rightarrow\mathbb{C}$ be the standard symmetric bilinear form determined by $(\alpha_{i},\alpha_{k})=\langle h_{i},\alpha_{k}\rangle$ for any $i,k\in I$.

The following lemma is very useful:
\begin{lem}[\cite{CB}]\label{lem Crawley-Boevey formula}
For any $\Lambda$-modules $X,Y$, the following holds:
\begin{equation}\label{equ Crawley-Boevey formula}
\dim\Ext_{\Lambda}^{1}(X,Y)=\dim\Hom_{\Lambda}(X,Y)+\dim\Hom_{\Lambda}(Y,X)-(\dimv X,\dimv Y).
\end{equation}
In particular, we have $\dim\Ext_{\Lambda}^{1}(X,Y)=\dim\Ext_{\Lambda}^{1}(Y,X)$.
\end{lem}

For any $i \in I$, let $S_{i}$ be the (one-dimensional) simple $\Lambda$-module concentrated at $i$, $I_{i}$ (resp. $P_{i}$) be the injective hull (resp. projective cover) of $S_{i}$.

For any $\Lambda$-module $X$, the \textit{$i$-socle} (resp. \textit{$i$-top}) of $X$ is the $S_{i}$-isotypic component of the socle (resp. top) of $X$, denoted by $\soc_{i}X$ (resp. $\tp_{i}X$). For a sequence $(j_{1},\ldots,j_{t})$ of indices with $1\leq j_{p}\leq n$ for all $p$, there is a unique chain
$$0=X_{0}\subseteq X_{1}\subseteq \cdots \subseteq X_{t}\subseteq X$$
of submodules of $X$ such that $X_{p}/X_{p-1}=\soc_{j_{p}}(X/X_{p-1})$. Define $\soc_{(j_{1},\ldots,j_{t})}(X)$ to be the submodule $X_{t}$ in the above chain.

\subsection{Geometric crystals}\label{subsec geometric crystals}
For any $\nu=\sum_{i\in I}\nu_{i}\alpha_{i}\in Q_{+}$, let $\Lambda(\nu)$ be the variety of all finite-dimensional $\Lambda$-modules with dimension vector $\nu$. It is an affine algebraic variety and the group $G(\nu)=\prod_{i\in I}\GL(\mathbb{C}^{\nu_{i}})$ acts on it by conjugation. Denote by $\Irr\Lambda(\nu)$ the set of irreducible components of $\Lambda(\nu)$. Let $\mathcal{B}=\bigsqcup_{\nu}\Irr\Lambda(\nu)$.

For any $Z\in\Irr\Lambda(\nu)$, we define the weight of $Z$ to be $\wt(Z)=-\nu\in Q_{-}$.

For any $i\in I$ and $c\in\mathbb{N}$, define
$$\Lambda(\nu)_{i,c}=\{X \in\Lambda(\nu)|\dim\tp_{i}X=c\}.$$
They are locally closed subsets and $\Lambda(\nu)=\bigsqcup_{c\in\mathbb{N}}\Lambda(\nu)_{i,c}$. Thus for any $Z\in\Irr\Lambda(\nu)$, there is a unique $c$ such that $Z\cap\Lambda(\nu)_{i,c}$ is open dense in $Z$. We then define two maps $\varepsilon_{i}(Z)=c$ and $\varphi_{i}(Z)=\varepsilon_{i}(Z)+\langle h_{i},\wt(Z)\rangle$.

It is also known that the map $Z\mapsto Z\cap\Lambda(\nu)_{i,c}$ gives a bijection
$$\{Z\in\Irr\Lambda(\nu)|\varepsilon_{i}(Z)=c\}\longleftrightarrow\Irr\Lambda(\nu)_{i,c}.$$

Denote by $\Omega(\nu,i,c)$ the set of triples $(X,Y,f)$ where $X\in\Lambda(\nu)_{i,0}$, $Y\in\Lambda(\nu+c\alpha_{i})_{i,c}$ and $f:X\rightarrow Y$ is an injective morphism of $\Lambda$-modules. We then have the following diagram
$$\Lambda(\nu)_{i,0}\xleftarrow{p}\Omega(\nu,i,c)\xrightarrow{q}\Lambda(\nu+c\alpha_{i})_{i,c},$$
where $p$ and $q$ are the first and second projections. Then $p$ is a locally trivial fibration with a smooth and connected fibre and $q$ is a principal $G(\nu)$-bundle. Hence the above diagram induces the following bijection
\begin{equation*}
\xymatrix{
\{Z\in\Irr\Lambda(\nu)|\varepsilon_{i}(Z)=0\} \ar @<0.5ex> [r]^-{\widetilde{f}_{i}^{c}}  & \{Z\in\Irr\Lambda(\nu+c\alpha_{i})|\varepsilon_{i}(Z)=c\} \ar @<0.5ex>[l]^-{\widetilde{e}_{i}^{\max}}
}
\end{equation*}

Then we define operators $\widetilde{e}_{i},\widetilde{f}_{i}:\mathcal{B}\rightarrow\mathcal{B}\cup\{0\}$ as follows:
$$\widetilde{e}_{i}(Z)=\widetilde{f}_{i}^{c-1}\widetilde{e}_{i}^{\max}(Z),\ \widetilde{f}_{i}(Z)=\widetilde{f}_{i}^{c+1}\widetilde{e}_{i}^{\max}(Z),\ \text{for any}\ Z\ \text{with}\ \varepsilon_{i}(Z)=c.$$

\begin{thm}[\cite{KS}]\label{thm KS crystal iso}
Equipped with the maps $\wt,\varepsilon_{i},\varphi_{i},\widetilde{e}_{i},\widetilde{f}_{i}$ defined above, the set $\mathcal{B}=\bigsqcup_{\nu\in Q_{+}}\Irr\Lambda(\nu)$ is a crystal and isomorphic to the crystal structure on the canonical basis $\mathbf{B}$.
\end{thm}

\subsection{MV polytopes via preprojective algebras}\label{subsec MV via prepro alg}
Let $\gamma$ be a chamber weight. Then there exist $i\in I$ and $w\in W$ such that $w$ has a reduced expression $w=s_{k_{1}}s_{k_{2}}\cdots s_{k_{s}}$ with $k_{s}=i$ and $\gamma=w\varpi_{i}$. We define a $\Lambda$-module
$$N(-w\varpi_{i})=\soc_{(k_{s},k_{s-1},\ldots,k_{1})}(I_{i}).$$
This is the unique (up to isomorphism) $\Lambda$-submodule of $I_{i}$ with dimension vector $\varpi_{i}-w\varpi_{i}$.

\begin{rem}
This definition is due to Gei\ss, Leclerc and Schr\"{o}er \cite{GLS2} while the notation $N(-w\varpi_{i})$ is after Baumann and Kamnitzer who alternatively defined this module using Nakajima's quiver and reflection functors \cite{BK}. A detailed proof of the equivalence between the two definitions can be found in \cite{J}. The choice of the definition and the notation will be convenient for us to prove our results.
\end{rem}

Note that $-w\varpi_{i}$ is also a chamber weight. The following useful lemma can be easily deduced using definitions and \cite{GLS2} Proposition 9.6.

\begin{lem}\label{lem property of N_gamma}
(1). For any $\gamma\in\Gamma^{j}$, $N(\gamma)$ has trivial $j$-top.

(2). For any $\gamma\in\Gamma_{j}$, $\dim\tp_{j}N(\gamma)=\langle h_{j},\gamma\rangle$. And we have the following short exact sequence:
$$0\rightarrow N(s_{j}\gamma)\rightarrow N(\gamma)\rightarrow S_{j}^{\oplus\langle h_{j},\gamma\rangle}\rightarrow 0.$$
\end{lem}

For any $\gamma\in\Gamma$, define an operator $D_{\gamma}:=\dim\Hom_{\Lambda}(N(\gamma),-)$. It is clear that for any $\nu\in\mathbb{N}^{n}$, $D_{\gamma}$ is a constructible function on $\Lambda(\nu)$. Thus for each irreducible component of $\Lambda(\nu)$, it takes a constant value on a dense open subset. For each $Z\in\Irr\Lambda(\nu)$, denote by $D_{\gamma}(Z)$ the generic value on $Z$. The following theorem tells us how to construct MV polytopes using representations of preprojective algebras.

\begin{thm}[\cite{BK}]\label{thm BK}
(i). For any $\gamma\in\Gamma$ and $Z\in\Irr\Lambda(\nu)$, set $M_{\gamma}=-D_{\gamma}(Z)$. Then $M_{\bullet}$ satisfies the edge inequalities and tropical Pl\"{u}cker relations. Hence $P(M_{\bullet})$ is an MV polytope.

(ii). The map $\Pol:\mathcal{B}\rightarrow\mathcal{MV}$ defined by $Z\mapsto P((-D_{\gamma}(Z))_{\gamma\in\Gamma})$ is a crystal isomorphism.
\end{thm}

\section{Proofs of the results: The simply-laced case}\label{sec proof simply-laced}
In this section we assume that $\mathfrak{g}$ is of simply-laced type. The following notations will be fixed throughout this section: Let $P=P(M_{\bullet})$ be an MV polytope and $P'=\tildef_{j}(P(M_{\bullet}))=P(M_{\bullet}')$ for $j\in I$. Let $Z=\Pol^{-1}(P)$. Then $Z'=\tildef_{j}(Z)=\Pol^{-1}(\tildef_{j}(P))$.

\subsection{The Proof of Theorem \ref{thm main simply-laced case} in the simply-laced case}\label{subsec pf of thm 1 sc}
By the definition of $\tildef_{j}$ on $\mathcal{B}$, there exist general points $T$ in $Z$ and $T'$ in $Z'$ with the following short exact sequence:
$$0\rightarrow T\rightarrow T'\rightarrow S_{j}\rightarrow 0.$$

Recall that $c_{j}(P)=M_{\varpi_{j}}-M_{s_{j}\varpi_{j}}-1$. We first show that $c_{j}(P)$ is closely related to the map $\varphi_{j}$ in the crystal structure:
\begin{lem}\label{lem express c_j}
For any $j\in I$, we have $c_{j}(P)=\varphi_{j}(Z)-1$.
\end{lem}
\begin{proof}
Since $N(\varpi_{j})=P_{j}$ (see \cite{BK}, Proposition 3.6), we have
\begin{equation}\label{equ 1}
D_{\varpi_{j}}(Z)=\dim\Hom_{\Lambda}(P_{j},T)=\dim T_{j}.
\end{equation}

By Lemma \ref{lem property of N_gamma}, we have the following short exact sequence
$$0\rightarrow N(s_{j}\varpi_{j})\rightarrow N(\varpi_{j})\rightarrow S_{j}\rightarrow 0.$$
Applying $\Hom_{\Lambda}(-,T)$ to the sequence, we have the following long exact sequence
$$0\rightarrow \Hom_{\Lambda}(S_{j},T) \rightarrow \Hom_{\Lambda}(P_{j},T) \rightarrow \Hom_{\Lambda}(N(s_{j}\varpi_{j}),T)\rightarrow \Ext_{\Lambda}^{1}(S_{j},T)\rightarrow 0.$$

Comparing dimensions of modules in the above sequence, we deduce that
$$D_{s_{j}\varpi_{j}}(Z)=\dim\Ext_{\Lambda}^{1}(S_{j},T)+\dim\Hom_{\Lambda}(P_{j},T)-\dim\Hom_{\Lambda}(S_{j},T).$$

Now using (\ref{equ Crawley-Boevey formula}), we have
\begin{equation}\label{equ 2}
\begin{split}
D_{s_{j}\varpi_{j}}(Z) & =\dim\Hom_{\Lambda}(T,S_{j})-(\dimv S_{j},\dimv T)+\dim T_{j}\\
                       & =\varepsilon_{j}(Z)+\langle h_{j},\wt(Z)\rangle+\dim T_{j}
\end{split}
\end{equation}

Using (\ref{equ 1}) and (\ref{equ 2}), we can deduce that
\begin{equation*}
\begin{split}
c_{j}(P)&=M_{\varpi_{j}}-M_{s_{j}\varpi_{j}}-1=-D_{\varpi_{j}}(Z)+D_{s_{j}\varpi_{j}}(Z)-1\\
        &=\varepsilon_{j}(Z)+\langle h_{j},\wt(Z)\rangle-1=\varphi_{j}(Z)-1.
\end{split}
\end{equation*}
\end{proof}

\begin{lem}\label{lem geq 1}
For any $\gamma\in\Gamma$, we have $D_{\gamma}(Z')\geq D_{\gamma}(Z)$.
\end{lem}
\begin{proof}
Since we have the exact sequence
$$0\rightarrow T\rightarrow T'\rightarrow S_{j}\rightarrow 0,$$
we know that $\Hom_{\Lambda}(N(\gamma),T)\hookrightarrow\Hom_{\Lambda}(N(\gamma),T')$.
\end{proof}

\begin{lem}\label{lem geq 2}
For $\gamma\in\Gamma_{j}$, we have $D_{\gamma}(Z')\geq D_{s_{j}\gamma}(Z)-c_{j}(P)\langle h_{j}, \gamma\rangle$.
\end{lem}
\begin{proof}
By Lemma \ref{lem property of N_gamma} we have the following short exact sequence:
\begin{equation}\label{equ 5}
0\rightarrow N(s_{j}\gamma)\rightarrow N(\gamma)\rightarrow S_{j}^{\oplus \langle h_{j},\gamma\rangle}\rightarrow 0.
\end{equation}

Applying $\Hom_{\Lambda}(-,T')$ to (\ref{equ 5}), we have the following long exact sequence
\begin{equation}\label{equ 3}
\begin{split}
0 \rightarrow \Hom_{\Lambda}(S_{j}^{\oplus \langle h_{j},\gamma\rangle},T')&\rightarrow \Hom_{\Lambda}(N(\gamma),T') \\
&\rightarrow \Hom_{\Lambda}(N(s_{j}\gamma),T')\rightarrow \Ext_{\Lambda}^{1}(S_{j}^{\oplus \langle h_{j},\gamma\rangle},T').
\end{split}
\end{equation}

Hence the following inequality holds:
\begin{equation*}
\begin{split}
\dim\Hom_{\Lambda}(N(\gamma),T')+&\dim\Ext_{\Lambda}^{1}(S_{j}^{\oplus \langle h_{j},\gamma\rangle},T')\geq\\
&\dim\Hom_{\Lambda}(S_{j}^{\oplus \langle h_{j},\gamma\rangle},T')+\dim\Hom_{\Lambda}(N(s_{j}\gamma),T').
\end{split}
\end{equation*}

Using (\ref{equ Crawley-Boevey formula}), we deduce that
\begin{equation*}
D_{\gamma}(Z')\geq D_{s_{j}\gamma}(Z')+\langle h_{j},\gamma\rangle(\dimv S_{j},\dimv T')-\dim\Hom_{\Lambda}(T',S_{j}^{\oplus \langle h_{j},\gamma\rangle}).
\end{equation*}

By Lemma \ref{lem geq 1} (b), we have $D_{s_{j}\gamma}(Z)=D_{s_{j}\gamma}(Z')$. It is clear that $\dimv T'=\dimv T+\alpha_{j}$, $\dim\tp_{j}(T')=\varepsilon_{j}(Z')=\varepsilon_{j}(Z)+1$. We then deduce that
\begin{equation*}
\begin{split}
D_{\gamma}(Z') &\geq D_{s_{j}\gamma}(Z)-(\langle h_{j},\wt(Z)-\alpha_{j}\rangle+\dim\tp_{j}(T'))\langle h_{j},\gamma\rangle\\
              &=D_{s_{j}\gamma}(Z)-(\langle h_{j},\wt(Z)\rangle-2+\varepsilon_{j}(Z)+1)\langle h_{j},\gamma\rangle\\
              &=D_{s_{j}\gamma}(Z)-(\varphi_{j}(Z)-1)\langle h_{j},\gamma\rangle.
\end{split}
\end{equation*}

By Lemma \ref{lem express c_j}, we have $D_{\gamma}(Z')\geq D_{s_{j}\gamma}(Z)-c_{j}(P)\langle h_{j},\gamma\rangle$.
\end{proof}

The following corollary follows directly from the proof above.
\begin{cor}\label{cor equ holds}
$D_{\gamma}(Z')=D_{s_{j}\gamma}(Z)-c_{j}(P)\langle h_{j},\gamma\rangle$ if and only if the last map in the sequence (\ref{equ 3}) is surjective.
\end{cor}

Combining Lemma \ref{lem geq 1} and Lemma \ref{lem geq 2}, we know that for any $\gamma\in\Gamma_{j}$,
$$D_{\gamma}(Z')\geq \max\{D_{\gamma}(Z),D_{s_{j}\gamma}(Z)-c_{j}(P)\langle h_{j}, \gamma\rangle\}.$$
Recall that $M_{\gamma}=-D_{\gamma}(Z)$. The above inequality is
$$M'_{\gamma}\leq \min\{M_{\gamma},M_{s_{j}\gamma}+c_{j}(P)\langle h_{j}, \gamma\rangle\}.$$
Hence the condition (c) in Lemma \ref{lem AM op in M} is satisfied for $P(M'_{\bullet})=\tildef_{j}(P)$. By Remark \ref{rem enough to prove main}, we have completed the proof of Theorem \ref{thm main simply-laced case} in the simply-laced case.

\subsection{The proof of Theorem \ref{thm minuscule} in the simply-laced case}\label{subsec pf of thm 2 sc}
Since $\langle h_{j},\gamma\rangle=1$, we know that $\gamma\in\Gamma_{j}$. By Lemma \ref{lem geq 1} and \ref{lem geq 2} we already have
\begin{equation}\label{equ 6}
D_{\gamma}(Z')\geq \max\{D_{\gamma}(Z),D_{s_{j}\gamma}(Z)-c_{j}(P)\}.
\end{equation}
We only need to prove that the equality holds in the above formula.

Applying $\Hom_{\Lambda}(-,T)$ and $\Hom_{\Lambda}(-,T')$ to the sequence (\ref{equ 5}), we have the following commutative diagram whose rows and columns are exact:
\begin{equation*}
\xymatrix{
            & 0 \ar[d]                                                & 0 \ar[d]                                            &  \\
  0  \ar[r] & \Hom_{\Lambda}(S_{j},T) \ar[d] \ar[r]                   & \Hom_{\Lambda}(S_{j},T')  \ar[d]                   &  \\
  0  \ar[r] & \Hom_{\Lambda}(N(\gamma),T) \ar[d] \ar[r]               & \Hom_{\Lambda}(N(\gamma),T')  \ar[d]               &  \\
            & \Hom_{\Lambda}(N(s_{j}\gamma),T) \ar[d]_{\alpha} \ar[r] & \Hom_{\Lambda}(N(s_{j}\gamma),T')  \ar[d]_{\beta}  &  \\
            & \Ext^{1}_{\Lambda}(S_{j},T) \ar[r]                      & \Ext^{1}_{\Lambda}(S_{j},T') \ar[r]                & \Ext^{1}_{\Lambda}(S_{j},S_{j})=0  }
\end{equation*}

If $\beta$ is surjective, we have $D_{\gamma}(Z')=D_{s_{j}\gamma}(Z)-c_{j}(P)$ by Corollary \ref{cor equ holds}.

Now assume that $\beta$ is not surjective, then $\alpha$ is not surjective either. We claim that this implies $D_{\gamma}(Z')=D_{\gamma}(Z)$, namely
$$\dim\Hom_{\Lambda}(N(\gamma),T)=\dim\Hom_{\Lambda}(N(\gamma),T').$$

Suppose on the contrary that $\dim\Hom_{\Lambda}(N(\gamma),T)<\dim\Hom_{\Lambda}(N(\gamma),T')$. Hence there exists $f\in\Hom_{\Lambda}(N(\gamma),T')$ such that the following diagram is commutative:
\begin{equation*}
\xymatrix{
  N(\gamma) \ar[d]_{f} \ar[r]^{p_{1}} & T'/T\simeq S_{j}       \\
  T' \ar[ur]_{p_{2}}                     }
\end{equation*}
Hence there exits $g\in\Hom_{\Lambda}(N(s_{j}\gamma),T)$ such that the following diagram is commutative with exact rows:
\begin{equation*}
\xymatrix{
  0 \ar[r] & N(s_{j}\gamma) \ar[d]_{g} \ar[r] & N(\gamma) \ar[d]_{f} \ar[r]^{p_{1}} & S_{j} \ar[d]_{\id} \ar[r] & 0  \\
  0 \ar[r] & T \ar[r] & T' \ar[r]^{p_{2}} & S_{j} \ar[r] & 0   }
\end{equation*}

Now let us consider any $\Lambda$-module $X$ which is an extension of $T$ by $S_{j}$. By the definition of $\tildef_{j}$ on the geometric crystal $\mathcal{B}$, we know that $X\in\tildef_{j}(Z)$. Since $T'$ is a general point in $\tildef_{j}(Z)$, using the well-known fact that the function $\Hom_{\Lambda}(N(\gamma),-):\tildef_{j}(Z)\rightarrow\mathbb{Z}$ is upper semi-continuous (see for example \cite{CS}), we have
$$\dim\Hom_{\Lambda}(N(\gamma),X)\geq\dim\Hom_{\Lambda}(N(\gamma),T')>\dim\Hom_{\Lambda}(N(\gamma),T).$$

The previous arguments for $T'$ can be applied to $X$. So there exists the following commutative diagram:
\begin{equation*}
\xymatrix{
  0 \ar[r] & N(s_{j}\gamma) \ar[d] \ar[r] & N(\gamma) \ar[d] \ar[r] & S_{j} \ar[d]_{\id} \ar[r] & 0  \\
  0 \ar[r] & T \ar[r] & X \ar[r] & S_{j} \ar[r] & 0   }
\end{equation*}
This implies that the map $\alpha$ is surjective, which contradicts to our assumption.

To summarize, we have proved the fact that $D_{\gamma}(Z')$ equals either $D_{\gamma}(Z)$ or $D_{s_{j}\gamma}(Z)-c_{j}(P)$. Together with (\ref{equ 6}), it completes the proof of Theorem \ref{thm minuscule} in the simply-laced case.

\subsection{The proof of Theorem \ref{thm simply-laced case} in the simply-laced case}\label{subsec pf of thm 3 sc}
We keep the notations in the previous subsections.

\begin{lem}
$\dim\Ext^{1}_{\Lambda}(T,S_{j})=D_{-s_{j}\varpi_{j}}(Z)-D_{\varpi_{j}}(Z)+D_{s_{j}\varpi_{j}}(Z).$
\end{lem}
\begin{proof}
We already know from the proof of Lemma \ref{lem express c_j} that
\begin{align*}
D_{\varpi_{j}}(Z)&=\dim T_{j},\\
D_{s_{j}\varpi_{j}}(Z)&=\varepsilon_{j}(Z)+\langle h_{j},\wt(Z)\rangle+\dim T_{j}.
\end{align*}

By definition $N(-s_{j}\varpi_{j})=S_{j}$. Hence $D_{-s_{j}\varpi_{j}}(Z)=\dim\Hom_{\Lambda}(S_{j},T)$.

Now using (\ref{equ Crawley-Boevey formula}), we deduce that
\begin{equation*}
\begin{split}
\dim\Ext^{1}_{\Lambda}(T,S_{j})&=\dim\Hom_{\Lambda}(T,S_{j})+\dim\Hom_{\Lambda}(S_{j},T)-(\dimv T,\dimv S_{j})\\
                               &=\varepsilon_{j}(Z)+\dim\Hom_{\Lambda}(S_{j},T)+\langle h_{j},\wt(Z)\rangle\\
                               &=D_{-s_{j}\varpi_{j}}(Z)-D_{\varpi_{j}}(Z)+D_{s_{j}\varpi_{j}}(Z).
\end{split}
\end{equation*}
Hence the proof is completed.
\end{proof}

By the above lemma, we have $m(j,P)=\dim\Ext^{1}_{\Lambda}(T,S_{j})\geq 0$.

Now let $n\geq m(j,P)$ be a positive integer, let $T^{(n)}$ be a general point in $Z^{(n)}=\tildef_{j}^{(n)}(Z)$ and $T^{(n+1)}$ be a general point in $Z^{(n+1)}=\tildef_{j}^{(n+1)}(Z)$ such that the short exact sequence holds:
$$0\rightarrow T^{(n)}\rightarrow T^{(n+1)}\rightarrow S_{j}\rightarrow 0.$$

By Theorem \ref{thm main simply-laced case}, which we have already proved in the simply-laced case, and Remark \ref{rem formula implies AM conj}, we see that (\ref{equ more explicit M'}) implies $AM_{j}(P^{(k)})=\tildef_{j}(P^{(k)})$. Thus to prove the remaining part of Theorem \ref{thm simply-laced case}, we only need to prove:
\begin{equation}\label{equ thm 3 final}
D_{\gamma}(Z^{(n+1)})=D_{s_{j}\gamma}(Z^{(n)})-c_{j}(P^{(n)})\langle h_{j},\gamma\rangle,\ \ \text{for any } \gamma\in\Gamma_{j}.
\end{equation}

Since $n\geq m(j,P)$, we have $\Ext^{1}_{\Lambda}(T^{(n+1)},S_{j})=0$. We then use the same argument as in the proof of Lemma \ref{lem geq 2}. The last map in sequence (\ref{equ 3}) is now just a zero map. By Corollary \ref{cor equ holds}, (\ref{equ thm 3 final}) holds as desired.

\begin{rem}
In the above proof we actually use the fact that along a $j$-string of the crystal $\mathcal{B}$, the action of $\tildef_{j}$ eventually becomes the direct sum with the simple module $S_{j}$. Equivalently, along a $j$-string of the crystal $\mathcal{MV}$, the action of $\tildef_{j}$ eventually becomes the Minkowski sum with the polytope $P(j)$, where $P(j)$ is the MV polytope with vertices $0$ and $h_{j}$. The latter is known to experts but seems not mentioned in the literature.
\end{rem}

\subsection{An example for type $D$}\label{subsec eg of type D}
Let $M''_{\bullet}=(M''_{\gamma})_{\gamma\in\Gamma}$ be given by (\ref{equ M'}). In this subsection we give an example of type $D_{4}$ in which $M'_{\bullet}\neq M''_{\bullet}$ but $\tildef_{j}(P)=AM_{j}(P)$ still holds.

Let $\mathfrak{g}$ be of type $D_{4}$ with Dynkin diagram
$$
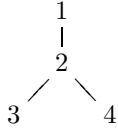
\xymatrix@-1.4pc{
&1 \ar@{-}[d] \\
&2 \ar@{-}[dl] \ar@{-}[dr]&\\
3 & & 4
}
$$
Let $j=1$ and we choose a particular $\gamma_{0}=-s_{1}s_{2}s_{4}s_{3}s_{2}\varpi_{2}\in\Gamma$. It is easy to check that $\langle\gamma_{0}, h_{1}\rangle=2$.

The modules $N(\gamma_{0})$ and $N(s_{1}\gamma_{0})$ can be visualized as follows:
$$
N(\gamma_{0}) 
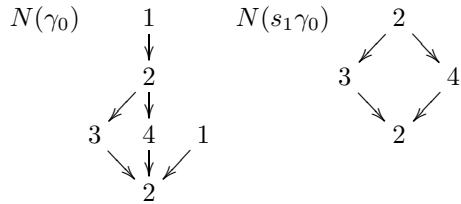
{
\xymatrix@-1.2pc{
         & 1 \ar[d]         &           \\
         & 2 \ar[dl] \ar[d] &           \\
3 \ar[dr]& 4 \ar[d]         & 1 \ar[dl] \\
         & 2                &
}
}\ \
N(s_{1}\gamma_{0}) {
\xymatrix@-1.2pc{
         & 2 \ar[dl] \ar[dr] &           \\
3 \ar[dr]&& 4 \ar[dl]                    \\
         & 2                &
}
}\noindent
$$

The variety $\Lambda(2\alpha_{2})$ contains just one point, namely $T=S_{2}\oplus S_{2}$. So $Z=\{T\}$ is the unique irreducible component. Let $P=P(M_{\bullet})=\Pol(Z)$ and $P'=P(M'_{\bullet})=\tildef_{1}(P)$. Denote by $X$ the unique extension (up to isomorphism) of $S_{1}$ by $S_{2}$. Then $T'=X\oplus S_{2}$ is a general point in $Z'=\tildef_{1}(Z)$. Then we have
\begin{align*}
M_{\gamma_{0}}&=-\dim\Hom_{\Lambda}(N(\gamma_{0}),T)=0,\\
M_{s_{1}\gamma_{0}}&=-\dim\Hom_{\Lambda}(N(s_{1}\gamma_{0}),T)=-2,\\
c_{1}(P)&=\varepsilon_{1}(Z)+\langle h_{1},\wt(Z)\rangle-1=0-\langle h_{1},2\alpha_{2}\rangle-1=1.
\end{align*}
So we have $M_{s_{1}\gamma_{0}}+c_{1}(P)\langle h_{1},\gamma_{0}\rangle=0$. Hence
$$M''_{\gamma_{0}}=\min\{M_{\gamma_{0}},M_{s_{1}\gamma_{0}}+c_{1}(P)\langle h_{1},\gamma_{0}\rangle\}=0.$$
However, $M_{\gamma_{0}}'=-\dim\Hom_{\Lambda}(N(\gamma_{0}),T')=-1\neq M''_{\gamma_{0}}$.

Next we will show that $AM_{j}(P)=\tildef_{j}(P)$.

Let $AM_{j}(P)=P(\tilde{M}_{\bullet})$. By Lemma \ref{lem AM op in M} and Theorem \ref{thm main simply-laced case}, we know that $M'_{\gamma}\leq\tilde{M}_{\gamma}\leq M''_{\gamma}$ for all $\gamma\in\Gamma$. And we need to prove $M'_{\gamma}=\tilde{M}_{\gamma}$.

In fact we have $M'_{\gamma}=\tilde{M}_{\gamma}=M''_{\gamma}$ for all $\gamma\neq\gamma_{0}$. The reason is that $\gamma_{0}$ is the only chamber weight such that $\langle\gamma, h_{1}\rangle=2$. To check this, one can find out all $\gamma\in\Gamma$ by exhausting all the possible modules $N(\gamma)$. Then by Theorem \ref{thm minuscule}, $M_{\gamma}'=M''_{\gamma}$ for all $\gamma\neq\gamma_{0}$.

Note that $M'_{\gamma_{0}}=-1$ and $M''_{\gamma_{0}}=0$. So we have $\tilde{M}_{\gamma_{0}}=M''_{\gamma_{0}}$ or $\tilde{M}_{\gamma_{0}}=M'_{\gamma_{0}}$. Now we claim that $M''_{\bullet}$ does not satisfy the edge inequalities, which implies $\tilde{M}_{\gamma_{0}}=M'_{\gamma_{0}}$ because $\tilde{M}_{\bullet}$ satisfies the edge inequalities.

Let $w=s_{1}s_{2}s_{4}s_{3}s_{2}$. So $\gamma_{0}=-w\varpi_{2}$. Note that $w_{0}\varpi_{k}=-\varpi_{k}$ for all $k\in I$ in case of $D_{4}$. We have $\gamma_{0}=ww_{0}\varpi_{2}$. Using the fact $w_{0}s_{2}=s_{2}w_{0}$, we have $-ws_{2}\varpi_{2}=ww_{0}s_{2}\varpi_{2}$. Therefore, the left hand side of the edge inequality (\ref{equ edge inequality}) applied for $M''_{\bullet}$ with $ww_{0}\in W$ and $i=2$ is the following:
\begin{equation*}
\begin{split}
& M''_{-w\varpi_{2}}+M''_{-ws_{2}\varpi_{2}}+\sum_{k\in\{1,3,4\}}c_{k2}M''_{-w\varpi_{k}}\\
=\ & M''_{-\gamma_{0}}+M''_{-s_{1}s_{2}\varpi_{2}}-M''_{-s_{1}\varpi_{1}}
-M''_{-s_{1}s_{2}s_{3}\varpi_{3}}-M''_{-s_{1}s_{2}s_{4}\varpi_{4}}
\end{split}
\end{equation*}
We already know that $M''_{-\gamma_{0}}=0$. It is easy to see that $M''_{-s_{1}\varpi_{1}}=0$, $M''_{-s_{1}s_{2}\varpi_{2}}=M''_{-s_{1}s_{2}s_{3}\varpi_{3}}=M''_{-s_{1}s_{2}s_{4}\varpi_{4}}=-1$. Hence
$$M''_{-w\varpi_{2}}+M''_{-ws_{2}\varpi_{2}}+\sum_{k\in\{1,3,4\}}c_{k2}M''_{-w\varpi_{k}}=1>0,$$
which verifies that $M''_{\bullet}$ does not satisfy the edge inequalities.

\section{Proofs of the results: The non-simply-laced case}\label{sec proof non-simply-laced}
Relying on the results in the simply-laced case, we will use diagram automorphisms to prove our results in the non-simply-laced case.

\subsection{Diagram automorphisms}\label{subsec diagram automorphisms}
We assume that $\mathfrak{g}$ is a simple Lie algebra of simply-laced type and keep the notations in Section \ref{sec MV polytopes and crystal}. Let $\sigma$ be an automorphism of the Dynkin diagram, i.e. $\sigma:I\rightarrow I$ is a bijection such that $c_{ij}=c_{\sigma(i)\sigma(j)}$. Let $k$ be the order of $\sigma$. It is well known that $\sigma$ induces a Lie algebra automorphism of $\mathfrak{g}$ by assigning the Chevalley generators $e_{i}$, $f_{i}$ and $h_{i}$ to $e_{\sigma(i)}$, $f_{\sigma(i)}$ and $h_{\sigma(i)}$ respectively. It is called a \textit{diagram automorphism} of $\mathfrak{g}$. The Cartan subalgebra $\mathfrak{h}$ is stable under $\sigma$, which induces $\sigma\in\GL(\mathfrak{h}^{\ast})$ by $\langle h,\sigma(\lambda)\rangle=\langle\sigma(h),\lambda\rangle$. Define
$$\mathfrak{g}^{\sigma}:=\{x\in\mathfrak{g}|\sigma(x)=x\},\quad \mathfrak{h}^{\sigma}:=\{h\in\mathfrak{h}|\sigma(h)=h\}.$$
Then it is well-known that $\mathfrak{g}^{\sigma}$ is also a simple Lie algebra (see for example \cite{Kac} Section 8.3). The Cartan subalgebra of $\mathfrak{g}^{\sigma}$ is $\mathfrak{h}^{\sigma}$.
Moreover, $\sigma$ induces a group automorphism of the Weyl group $W$ by $\sigma(s_{i})=s_{\sigma(i)}$. Set $W^{\sigma}:=\{w\in W|\sigma(w)=w\}$. $\mathfrak{h}^{\sigma}$ is stable under the action of $W^{\sigma}$.

\begin{figure}\label{fig diagram automorphism}
$$
\xymatrix@-1.2pc{
&& 1\ar@{-}[rr] \ar@{-->}@/_0.5pc/[dd] && 2\ar@{-}[rr] \ar@{-->}@/_0.5pc/[dd] && \cdots \ar@{-}[rr] && l-1\ar@{-}[drr] \ar@{-->}@/_0.5pc/[dd] && \\
A_{2l-1} (l\geq 2): && && && && && l\\
&& 2l-1 \ar@{-}[rr] \ar@{-->}@/_0.5pc/[uu] && 2l-2 \ar@{-}[rr] \ar@{-->}@/_0.5pc/[uu] && \cdots \ar@{-}[rr] && l+2 \ar@{-}[urr] \ar@{-->}@/_0.5pc/[uu] &&
}
$$

$$
\xymatrix@-1.2pc{
&& && && && && l \ar@/_0.5pc/@{-->}[dd]\\
D_{l+1} (l\geq 3): && 1 \ar@{-}[rr] && 2\ar@{-}[rr] && \cdots\ar@{-}[rr] && l-1\ar@{-}[urr] \ar@{-}[drr] &&\\
&& && && && && l+1 \ar@/_0.5pc/@{-->}[uu]
}
$$

$$
\xymatrix@-1.2pc{
&& &1 \ar@{-}[d] \ar@/_1pc/@{-->}[ddl]\\
D_{4}: && &2 \ar@{-}[dl] \ar@{-}[dr]&\\
&& 3 \ar@/_1pc/@{-->}[rr] & & 4 \ar@/_1pc/@{-->}[uul]
}
$$

$$
\xymatrix@-1.2pc{
&& 1\ar@{-}[rr] \ar@{-->}@/_0.5pc/[dd] && 2\ar@{-}[drr] \ar@{-->}@/_0.5pc/[dd] && && \\
E_{6}: && && && 3 \ar@{-}[rr] && 4 \\
&& 6 \ar@{-}[rr] \ar@{-->}@/_0.5pc/[uu] && 5 \ar@{-}[urr] \ar@{-->}@/_0.5pc/[uu] && &&
}
$$

\caption{Diagram automorphisms of $\mathfrak{g}$ considered in this paper}
\end{figure}
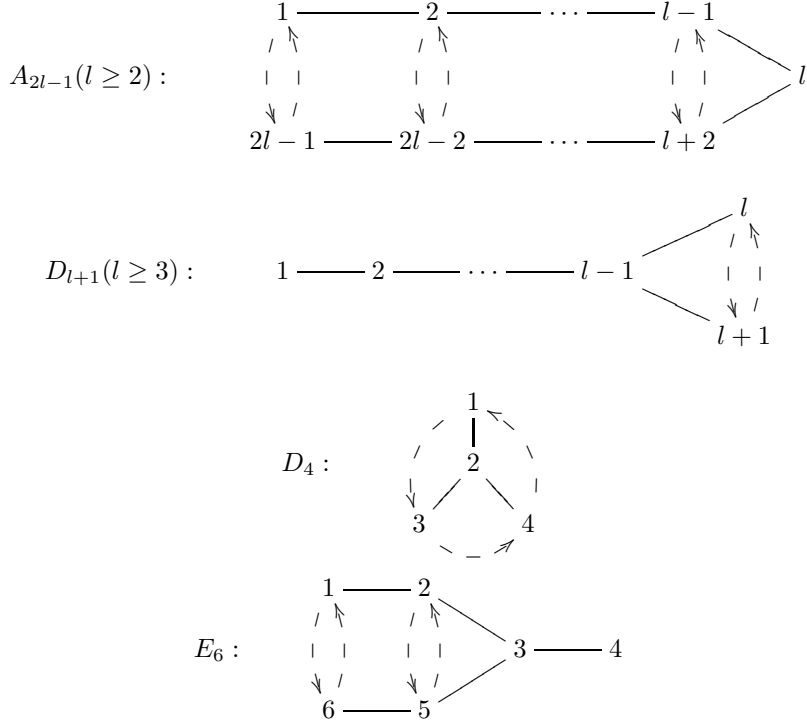

Throughout this section, we will only consider diagram automorphisms $\sigma$ illustrated in Figure \ref{fig diagram automorphism}, where the dashed arrows indicate the action of $\sigma$ on $I$ and vertices not connected with any arrows are invariant under $\sigma$. Then the fixed point subalgebra $\mathfrak{g}^{\sigma}$ is a simple Lie algebra of type $C_{l}$, $B_{l}$, $G_{2}$ and $F_{4}$ respectively. Note that we have covered all non-simply-laced types.

From now on we set $\hatfrakg=\mathfrak{g}^{\sigma}$. Let $\hatC=(\hat{c}_{ij})$ be the Cartan matrix of $\hatfrakg$ with index set $\hatI=\{1,2,\ldots,l\}$. It is convenient to treat $\hatI$ as a subset of $I=\{1,2,\ldots,n\}$. Let $\hatW$ be the Weyl group associated with $\hatfrakg$ with $\hats_{i}$ ($i\in\hatI$) the simple reflections. Let $\hat{\varpi_{i}}$ ($i\in\hatI$) be the fundamental weights for $\hatfrakg$ and set $$\hat{\Gamma}=\{\hatw\hat{\varpi_{i}}|\hatw\in\hatW,i\in\hatI\}.$$

Let $k$ be the order of $\sigma$. From Figure \ref{fig diagram automorphism} we see that $k=3$ if $\mathfrak{g}$ is of type $D_{4}$ and $\sigma$ is given by $\sigma(1)=3,\sigma(3)=4,\sigma(4)=1$, and $k=2$ in all remaining cases.

Let $k_{i}$ be the number of elements in the $\sigma$-orbit of $i\in I$. Again from Figure \ref{fig diagram automorphism} we can see that $k_{i}\in\{1,2,3\}$ for any $i$ and the $\sigma$-orbit of $i\in I$ is isomorphic to $A_{1}\times \cdots\times A_{1}$ ($k_{i}$ copies). It is well known that there exists a group isomorphism $\Theta:\hatW\simeq W^{\sigma}$ such that $\Theta(\hats_{i})=s_{i}^{\sigma}$ for all $i\in\hatI$ (see for example \cite{FRS}), where
\begin{equation*}
s_{i}^{\sigma}=\prod_{t=0}^{k_{i}-1}s_{\sigma^{t}(i)}.
\end{equation*}
Note that the product is independent of the order of $s_{\sigma^{t}(i)}$ for different $t$.

The following basic results will be used later.
\begin{lem}\label{lem basis res non-simply-laced}
(a). The Chevalley generators of $\mathfrak{h}^{\sigma}$ are given by $\hat{h}_{i}=h_{i}^{\sigma}=\sum_{t=0}^{k_{i}-1}h_{\sigma^{t}(i)}$, for any $i\in\hatI$.

(b). The fundamental weight $\hatpi_{i}$ is the restriction of $\varpi_{i}$ to the subspace $\mathfrak{h}^{\sigma}$.

(c). Let $\hatw\in\hatW$ and $w=\Theta(\hatw)\in W^{\sigma}$. For any $i\in\hatI$, we have
$$\hats_{i}\hatw<\hatw \Longleftrightarrow s_{i}^{\sigma}w<w \Longleftrightarrow s_{i}w<w,s_{\sigma(i)}w<w,\ldots,s_{\sigma^{k_{i}-1}(i)}w<w.$$

(d). For any $\hatw\in\hatW$ and $h\in\mathfrak{h}^{\sigma}$, we have $\hatw h=\Theta(\hatw)h$.
\end{lem}
\begin{proof}
(a) and (b) are well-known results. (c) and (d) can be found in \cite[Remark 2.3.2(3)]{NS}.
\end{proof}

\subsection{Diagram automorphisms and MV polytopes}\label{subsec diag automorphism and MV}
In this subsection we describe the relations between $\sigma$-invariant MV polytopes for $\mathfrak{g}$ and MV polytopes for $\hat{\mathfrak{g}}$. For details we refer to \cite{H} and \cite{NS}.

A diagram automorphism $\sigma$ induces an action on the set $\mathcal{MV}$ of MV polytopes as follows. Let $P=P(\mu_{\bullet})=P(M_{\bullet})$ be an MV polytope for $\mathfrak{g}$. $\sigma(P):=P(\mu_{\bullet}')$, where $\mu_{w}'=\sigma^{-1}(\mu_{\sigma(w)})$ for $w\in W$. Assuming that the BZ datum of $\sigma(P)$ is $M_{\bullet}'$, we have $M_{\gamma}'=M_{\sigma(\gamma)}$ for any $\gamma\in\Gamma$. We then define
$$\mathcal{MV}^{\sigma}=\{P\in\mathcal{MV}|\sigma(P)=P\}.$$

We know that an MV polytpe $P\in\mathcal{MV}^{\sigma}$ if and only if $\sigma(\mu_{w})=\mu_{\sigma(w)}$ for all $w\in W$, if and only if $M_{\gamma}=M_{\sigma(\gamma)}$ for all $\gamma\in\Gamma$.

Let $\hatMV$ be the set of MV polytopes for $\hatfrakg$. A polytope in $\hatMV$ will be denoted by $\hatP=\hatP(\hatM_{\bullet})=\hatP(\hatmu_{\bullet})$. The Kashiwara operators on $\hatMV$ are denoted by $\hat{f}_{j}$ for any $j\in\hatI$.

Let $P=P(\mu_{\bullet})$ be an MV polytope for $\mathfrak{g}$. Let $\Phi(\mu_{\bullet})$ be a collection of elements $(\hatmu_{\hatw})_{\hatw\in\hatW}$ in $\mathfrak{h}^{\sigma}\cap\mathbb{R}$ given by $\hatmu_{\hatw}=\mu_{\Theta(\hatw)}$. Then we define a map $\Phi:\mathcal{MV}^{\sigma}\rightarrow\hatMV$ as $\Phi(P(\mu_{\bullet}))=\hatP(\Phi(\mu_{\bullet}))$.

For any $i\in\hatI$, we define operators $\tildef_{i}^{\sigma}$ on $\mathcal{MV}$ as
\begin{equation*}
\tildef_{i}^{\sigma}=\prod_{t=0}^{k_{i}-1}\tildef_{\sigma^{t}(i)}.
\end{equation*}
Note that this definition does not depend on the order of $\tildef_{\sigma^{t}(i)}$ for different $t$.

\begin{thm}\label{thm sigma_inv_MV and hat MV}
The map $\Phi:\mathcal{MV}^{\sigma}\rightarrow\hatMV$ is a bijection such that $\Phi\circ\tildef_{j}^{\sigma}=\hat{f}_{j}\circ\Phi$ for all $j\in\hatI$.
\end{thm}
\begin{proof}
It was proved in \cite{H} that $\Phi$ is a bijection. The fact that $\Phi$ commutes with the Kashiwara operators was proved in \cite{NS} for $\mathfrak{g}$ of type $A$. For type $D$ and $E$ it can be proved in the same way. In fact, one only need to notice that Lemma 2.6.1, Lemma 2.7.2, Proposition 2.7.3 and Proposition 2.7.4 in \cite{NS} all hold in general.
\end{proof}

\subsection{Proof of Theorem \ref{thm main simply-laced case}: The non-simply-laced case}\label{subsec pf of thm 1 nsc}
Let $\hatP=\hatP(\hatM_{\bullet})$ be an MV polytope for $\hatfrakg$. Fix $j\in\hatI$ and let $\hatP'=\hatf_{j}(\hatP)=\hatP(\hatM'_{\bullet})$. Let $P=P(M_{\bullet})=\Phi^{-1}(\hatP)$ and $P'=P(M'_{\bullet})=\Phi^{-1}(\hatP')$. By Theorem \ref{thm sigma_inv_MV and hat MV} we have $P'=\tildef_{j}^{\sigma}(P)$.

\begin{lem}\label{lem M hat}
For any $i\in\hatI$ and $\hatw\in\hatW$, we have $\hatM_{\hatw\hatpi_{i}}=M_{\Theta(\hatw)\varpi_{i}}$.
\end{lem}
\begin{proof}
See \cite[Remark 2.5.2]{NS}.
\end{proof}

\begin{lem}\label{lem hat cj P}
For any $i\in\hatI$, we have $c_{i}(\hatP)=c_{i}(P)$.
\end{lem}
\begin{proof}
By definition, $c_{i}(\hatP)= \hatM_{\hatpi_{i}}-\hatM_{\hats_{i}\hatpi_{i}}-1$. We have that $\hatM_{\hatpi_{i}}=M_{\varpi_{i}}$ and $\hatM_{\hats_{i}\hatpi_{i}}=M_{s_{i}^{\sigma}\varpi_{i}}$ (Lemma \ref{lem M hat}). By definition, we have $s_{i}^{\sigma}=s_{i}s_{\sigma(i)}\cdots s_{\sigma^{k_{i}-1}(i)}$. Note that $\sigma(i)$, $\ldots$, $\sigma^{k_{i}-1}(i)$ are not connected with $i$ in the Dynkin graph. Hence we have $M_{s_{i}^{\sigma}\varpi_{i}}=M_{s_{i}\varpi_{i}}$. Therefore we have
$c_{i}(\hatP)= M_{\varpi_{i}}-M_{s_{i}\varpi_{i}}-1=c_{i}(P)$.
\end{proof}

\begin{lem}\label{lem non simply lace key}
For any $\hatgamma\in\hat{\Gamma}_{j}$, we have
\begin{equation}\label{equ main thm formula non-simply-laced}
\hatM'_{\hatgamma}\leq\min\{\hatM_{\hatgamma},\hatM_{\hats_{j}\hatgamma}+c_{j}(\hatP)\langle\hath_{j},\hatgamma\rangle\}.
\end{equation}
\end{lem}
\begin{proof}
Since $\hatgamma\in\hat{\Gamma}_{j}$, we have $\hatgamma=\hatw\hatpi_{i}$ where $\hats_{j}\hatw>\hatw$. Let $w=\Theta(w)$ and $\gamma=w\varpi_{i}$. Hence $\langle h_{j}^{\sigma},\gamma\rangle=\langle \hath_{j},\hatgamma\rangle>0$. By Lemma \ref{lem basis res non-simply-laced} (c), we deduce that $\langle h_{\sigma^{t}(j)},\gamma\rangle\geq 0$ for any $0\leq t\leq k_{j}-1$.

By Lemma \ref{lem basis res non-simply-laced}, \ref{lem M hat} and \ref{lem hat cj P}, we can deduce that (\ref{equ main thm formula non-simply-laced}) is equivalent to
\begin{equation}\label{equ alternative main nsl}
M'_{\gamma}\leq\min\{M_{\gamma},M_{s^{\sigma}_{j}\gamma}+c_{j}(P)\langle h^{\sigma}_{j},\gamma\rangle\}.
\end{equation}

Hence we can return to the simply-laced case and only need to prove (\ref{equ alternative main nsl}). In the following we only consider the case $k_{j}=2$. The case $k_{j}=3$ can be treated in the same way.

Recall that $f_{j}^{\sigma}=f_{j}f_{\sigma(j)}$. So $P'=f_{j}^{\sigma}(P)=f_{j}f_{\sigma(j)}(P)$. Let $P^{(1)}=P(M^{(1)}_{\bullet})=f_{\sigma(j)}(P)$. Then $P'=f_{j}(P^{(1)})$.

Since we have proved Theorem \ref{thm main simply-laced case} in the simply-laced case, we have
\begin{align*}
M'_{\gamma}& =M^{(1)}_{\gamma}\ \ \ \text{if } \langle h_{j},\gamma\rangle=0,\\
M'_{\gamma}& \leq\min\{M^{(1)}_{\gamma},M^{(1)}_{s_{j}\gamma}+c_{j}(P^{(1)})\langle h_{j},\gamma\rangle\}\ \ \ \text{if } \langle h_{j},\gamma\rangle>0.
\end{align*}

Recall $h_{j}^{\sigma}=h_{j}+h_{\sigma(j)}$ and note that $\langle h_{j}^{\sigma},\gamma\rangle > 0$. We need to consider the following three cases:

Case 1: $\langle h_{j},\gamma\rangle=0$ and $\langle h_{\sigma(j)},\gamma\rangle>0$. Now we have $M'_{\gamma}=M^{(1)}_{\gamma}$. Again by Theorem \ref{thm main simply-laced case} in the simply-laced case, we have
$$M^{(1)}_{\gamma}\leq\min\{M_{\gamma},M_{s_{\sigma(j)}\gamma}+c_{\sigma(j)}(P)\langle h_{\sigma(j)},\gamma\rangle\}.$$ We claim that $c_{\sigma(j)}(P)=c_{j}(P)$. In fact, $P\in\mathcal{MV}^{\sigma}$ implies that $M_{\gamma}=M_{\sigma(\gamma)}$ for all $\gamma\in\Gamma$. Hence we have
\begin{equation*}
c_{\sigma(j)}(P)=M_{\varpi_{\sigma(j)}}-M_{s_{\sigma(j)}\varpi_{\sigma(j)}}-1=M_{\varpi_{j}}-M_{s_{j}\varpi_{j}}-1=c_{j}(P).
\end{equation*}
Since $\langle h_{j},\gamma\rangle=0$, it is easy to see that $s_{\sigma(j)}\gamma=s_{j}^{\sigma}\gamma$, $\langle h_{\sigma(j)},\gamma\rangle=\langle h_{j}^{\sigma},\gamma\rangle$. Therefore (\ref{equ alternative main nsl}) holds in this case.

Case 2: $\langle h_{j},\gamma\rangle>0$ and $\langle h_{\sigma(j)},\gamma\rangle=0$. Now we have
$$M'_{\gamma} \leq\min\{M^{(1)}_{\gamma},M^{(1)}_{s_{j}\gamma}+c_{j}(P^{(1)})\langle h_{j},\gamma\rangle\}$$
We claim that $c_{j}(P^{(1)})=c_{j}(P)$. In fact, since $P^{(1)}=f_{\sigma(j)}(P)$ and $\varpi_{j},s_{j}\varpi_{j}\in\Gamma^{\sigma(j)}$, we have
\begin{equation*}
c_{j}(P^{(1)})=M^{(1)}_{\varpi_{j}}-M^{(1)}_{s_{j}\varpi_{j}}-1=M_{\varpi_{j}}-M_{s_{j}\varpi_{j}}-1=c_{j}(P).
\end{equation*}
Since $\langle h_{\sigma(j)},\gamma\rangle=0$, we have $\langle h_{\sigma(j)},s_{j}\gamma\rangle=0$ and $\langle h_{j},\gamma\rangle=\langle h_{j}^{\sigma},\gamma\rangle$. Hence $M^{(1)}_{\gamma}=M_{\gamma}$, $M^{(1)}_{s_{j}\gamma}=M_{s_{j}\gamma}$. Therefore (\ref{equ alternative main nsl}) holds in this case.

Case 3: $\langle h_{j},\gamma\rangle>0$ and $\langle h_{\sigma(j)},\gamma\rangle>0$.
Again by the theorem in the simply-laced case, we have
\begin{align*}
M^{(1)}_{\gamma}&\leq\min\{M_{\gamma},M_{s_{\sigma(j)}\gamma}+c_{\sigma(j)}(P)\langle h_{\sigma(j)},\gamma\rangle\},\\
M^{(1)}_{s_{j}\gamma}&\leq\min\{M_{s_{j}\gamma},M_{s_{\sigma(j)}s_{j}\gamma}+c_{\sigma(j)}(P)\langle h_{\sigma(j)},s_{j}\gamma\rangle\}.
\end{align*}
We can deduce that
\begin{equation}\label{equ non-sim 1}
M'_{\gamma}\leq\min\{M_{\gamma},M_{s_{\sigma(j)}s_{j}\gamma}+c_{\sigma(j)}(P)\langle h_{\sigma(j)},s_{j}\gamma\rangle
+c_{j}(P^{(1)})\langle h_{j},\gamma\rangle\}.
\end{equation}
As in the previous two cases we have $c_{j}(P^{(1)})=c_{j}(P)$ and $c_{\sigma(j)}(P)=c_{j}(P)$.
Note that $\langle h_{\sigma(j)},s_{j}\gamma\rangle=\langle h_{\sigma(j)},\gamma\rangle$. So (\ref{equ non-sim 1}) can be simplified as
\begin{equation}\label{equ non-sim 2}
M'_{\gamma}\leq\min\{M_{\gamma},M_{s_{\sigma(j)}s_{j}\gamma}+c_{j}(P)\langle h_{\sigma(j)}+h_{j},\gamma\rangle\},
\end{equation}
which is the same as (\ref{equ alternative main nsl}).
\end{proof}

In view of Remark \ref{rem enough to prove main}, the proof of Theorem \ref{thm main simply-laced case} is now completed.

\subsection{Proof of Theorem \ref{thm minuscule} in the non-simply-laced case}\label{subsec pf of thm 2 nsc}
We keep the notations in the previous subsection. Assuming that $\langle \hath_{j},\hatgamma\rangle=\langle h_{j}^{\sigma},\hatgamma\rangle=1$, we need to prove
\begin{equation}\label{equ non-sim thm 2}
\hatM'_{\hatgamma}=\min\{\hatM_{\hatgamma},\hatM_{\hats_{j}\hatgamma}+c_{j}(\hatP)\}.
\end{equation}

Since $\langle \hath_{j},\hatgamma\rangle=1$, we know that $\hatgamma\in\hat{\Gamma}_{j}$. Hence  $\hatgamma=\hatw\hatpi_{i}$ where $\hats_{j}\hatw>\hatw$. Let $w=\Theta(w)$ and $\gamma=w\varpi_{i}$. By the same reason as in the proof of Lemma \ref{lem non simply lace key}, (\ref{equ non-sim thm 2}) is equivalent to
\begin{equation}\label{equ non-sim thm 2 alternative}
M'_{\gamma}=\min\{M_{\gamma},M_{s_{j}^{\sigma}\gamma}+c_{j}(P)\}.
\end{equation}

By Lemma \ref{lem basis res non-simply-laced} (c) we know that $\langle h_{\sigma^{t}(j)},\gamma\rangle\geq 0$ for any $0\leq t\leq k_{j}-1$. Since $h_{j}^{\sigma}=\sum_{t=0}^{k_{i}-1}h_{\sigma^{t}(i)}$ and $\langle h_{j}^{\sigma},\hatgamma\rangle=1$, we deduce that there exist a unique $t_{1}$ such that $\langle h_{\sigma^{t_{1}}(j)},\gamma\rangle=1$ and for all other $t$, $\langle h_{\sigma^{t}(j)},\gamma\rangle=0$.

Now we are again in the simply-laced case. In the following we assume $k_{j}=2$. The case $k_{j}=3$ can be proved similarly.

Recall that $P'=f_{j}^{\sigma}(P)=f_{j}f_{\sigma(j)}(P)$. As in the last subsection, Let $P^{(1)}=P(M^{(1)}_{\bullet})=f_{\sigma(j)}(P)$. So we have $P'=f_{j}(P^{(1)})$.

Since we have proved the theorem in the simply-laced case, we know that
\begin{equation*}
M'_{\gamma}=\begin{cases}
            M^{(1)}_{\gamma} & \text{if } \langle h_{j},\gamma\rangle=0,\\
            \min\{M^{(1)}_{\gamma},M^{(1)}_{s_{j}\gamma}+c_{j}(P^{(1)})\} & \text{if } \langle h_{j},\gamma\rangle=1.
            \end{cases}
\end{equation*}

If $\langle h_{j},\gamma\rangle=0$, then $\langle h_{\sigma(j)},\gamma\rangle=1$. Again by the theorem in the simply-laced case, we have
$$M^{(1)}_{\gamma}=\min\{M_{\gamma},M_{s_{\sigma(j)}\gamma}+c_{\sigma(j)}(P)\}.$$
Note that $s_{\sigma(j)}\gamma=s_{j}^{\sigma}\gamma$ because $s_{j}\gamma=\gamma$. So $M_{s_{\sigma(j)}\gamma}=M_{s_{j}^{\sigma}\gamma}$. In the previous subsection we have proved that $c_{\sigma(j)}(P)=c_{j}(P)$. Thus (\ref{equ non-sim thm 2 alternative}) is proved.

If $\langle h_{j},\gamma\rangle=1$, then $\langle h_{\sigma(j)},\gamma\rangle=0$ and $\langle h_{\sigma(j)},s_{j}\gamma\rangle=0$. In this case we have $M^{(1)}_{\gamma}=M_{\gamma}$, $M^{(1)}_{s_{j}\gamma}=M_{s_{j}\gamma}=M_{s_{j}^{\sigma}\gamma}$ and $c_{j}(P^{(1)})=c_{j}(P)$ (proved in the previous subsection). Hence (\ref{equ non-sim thm 2 alternative}) still holds.

\subsection{Proof of Theorem \ref{thm simply-laced case} in the non-simply-laced case}\label{subsec pf of thm 3 nsc}
Let $\hatP=\hatP(\hatM_{\bullet})$ be an MV polytope for $\hatfrakg$ and $j\in\hatI$. Let us first prove
\begin{equation}\label{equ non-sim thm 3 a}
m(j,\hatP):=\hatM_{\hatpi_{j}}-\hatM_{-\hats_{j}\hatpi_{j}}-\hatM_{\hats_{j}\hatpi_{j}}\geq 0.
\end{equation}

Let $P=P(M_{\bullet})=\Phi^{-1}(\hatP)$ which is an MV polytope for $\mathfrak{g}$. By Lemma \ref{lem M hat}, we know that $\hatM_{\hatpi_{j}}=M_{\varpi_{j}}$, $\hatM_{-\hats_{j}\hatpi_{j}}=M_{-s_{j}^{\sigma}\varpi_{j}}$ and $\hatM_{\hats_{j}\hatpi_{j}}=M_{s_{j}^{\sigma}\varpi_{j}}$. Since $s_{j}^{\sigma}\varpi_{j}=s_{j}\varpi_{j}$, we have
\begin{equation}
m(j,\hatP)=m(j,P)=M_{\varpi_{j}}-M_{-s_{j}\varpi_{j}}-M_{s_{j}\varpi_{j}},
\end{equation}
which is positive by the theorem in the simply-laced case.

Now assume that $k$ is an integer such that $k>m(j,\hatP)$. Let $\hatP^{(k)}=\hatP(\hatM^{(k)}_{\bullet})=\hatf_{j}^{k}(\hatP)$ and $\hatP^{(k+1)}=\hatP(\hatM^{(k+1)}_{\bullet})=\hatf_{j}^{k+1}(\hatP)$. We are going to prove
\begin{equation}\label{equ non-sim thm3 b}
\hatM^{(k+1)}_{\hatgamma}=\begin{cases}
            & \hatM^{(k)}_{\hatgamma},\ \ \ \text{if } \hatgamma\in\hat{\Gamma}^{j}\\
            & \hatM^{(k)}_{\hats_{j}\hatgamma}+c_{j}(\hatP^{(k)})\langle \hath_{j},\hatgamma\rangle, \ \ \ \text{if } \hatgamma\in\hat{\Gamma}_{j}.
            \end{cases}
\end{equation}

We only need to prove for the case $\hatgamma\in\hat{\Gamma}_{j}$. We have $\hatgamma=\hatw\hatpi_{i}$ where $\hats_{j}\hatw>\hatw$. Let $w=\Theta(w)$ and $\gamma=w\varpi_{i}$. Then, as in the previous sections, we just need to prove
\begin{equation}\label{equ non-sim thm 3 c}
M^{(k+1)}_{\gamma}=M^{(k)}_{s_{j}\gamma}+c_{j}(P^{(k)})\langle h_{j}^{\sigma},\gamma\rangle.
\end{equation}

In the following we assume $k_{j}=2$. The case $k_{j}=3$ is similar. Let $P^{(k+1)}=P(M^{(k+1)}_{\bullet})=\Phi^{-1}(\hatP^{(k+1)})=(\tildef_{j}^{\sigma})^{k+1}(P)=\tildef_{j}^{k+1}\tildef_{\sigma(j)}^{k+1}(P)$. Let $P^{(k)}=P(M^{(k)}_{\bullet})=\Phi^{-1}(\hatP^{(k)})=(\tildef_{j}^{\sigma})^{k}(P)$. So $P^{(k+1)}=\tildef_{j}^{\sigma}(P^{(k)})=\tildef_{j}\tildef_{\sigma(j)}(P^{(k)})$. Let $P''=P(M''_{\bullet})=\tildef_{\sigma(j)}(P^{(k)})$, then $P^{(k+1)}=\tildef_{j}(P'')$.

We know that $P^{(k)}=\tildef_{j}^{k}\tildef_{\sigma(j)}^{k}(P)$. Let $\bar{P}=P(\bar{M}_{\bullet})=\tildef_{j}^{k}(P)$, hence $P^{(k)}=\tildef_{\sigma(j)}^{k}(\bar{P})$.

We claim that $m(\sigma(j),\bar{P})=m(j,P)$. In fact, we have
$$\langle h_{j},\varpi_{\sigma(j)}\rangle=\langle h_{j},-s_{\sigma(j)}\varpi_{\sigma(j)}\rangle=\langle h_{j},s_{\sigma(j)}\varpi_{\sigma(j)}\rangle=0.$$
Noting that $\bar{P}=\tildef_{j}^{k}(P)$, we have $\bar{M}_{\varpi_{\sigma(j)}}=M_{\varpi_{\sigma(j)}}$, $\bar{M}_{-s_{\sigma(j)}\varpi_{\sigma(j)}}=M_{-s_{\sigma(j)}\varpi_{\sigma(j)}}$ and $\bar{M}_{s_{\sigma(j)}\varpi_{\sigma(j)}}=M_{s_{\sigma(j)}\varpi_{\sigma(j)}}$. Thus we deduce that
\begin{equation*}
\begin{split}
m(\sigma(j),\bar{P})&=\bar{M}_{\varpi_{\sigma(j)}}-\bar{M}_{-s_{\sigma(j)}\varpi_{\sigma(j)}}-\bar{M}_{s_{\sigma(j)}\varpi_{\sigma(j)}}\\
&=M_{\varpi_{\sigma(j)}}-M_{-s_{\sigma(j)}\varpi_{\sigma(j)}}-M_{s_{\sigma(j)}\varpi_{\sigma(j)}}\\
&=m(\sigma(j),P)=m(j,P),
\end{split}
\end{equation*}
where the last equality holds because $P\in\mathcal{MV}^{\sigma}$.

Applying the theorem in the simply-laced case for $\bar{P}$ and $\sigma(j)$, we have
\begin{equation}\label{equ non-sim thm 3 d}
M''_{\gamma}=\begin{cases}
                   & M^{(k)}_{\gamma}\ \ \text{ if } \gamma\in\Gamma^{\sigma(j)}, \\
                   & M^{(k)}_{s_{\sigma(j)}\gamma}+c_{\sigma(j)}(P^{(k)})\langle h_{\sigma(j)},\gamma\rangle \ \ \text{ if } \gamma\in\Gamma_{\sigma(j)}.
                  \end{cases}
\end{equation}

Let $\tilde{P}=P(\tilde{M}_{\bullet})=\tildef_{\sigma(j)}^{k+1}(P)$. So $P''=\tildef_{j}^{k}(\tilde{P})$.

As above we can prove that $m(j,\tilde{P})=m(\sigma(j),P)=m(j,P)$. Then we apply the theorem in the simply-laced case for $\tilde{P}$ and $j$ and deduce that
\begin{equation}\label{equ non-sim thm 3 e}
M^{(k+1)}_{\gamma}=\begin{cases}
                   & M''_{\gamma}\ \ \text{ if } \gamma\in\Gamma^{j}, \\
                   & M''_{s_{j}\gamma}+c_{j}(P'')\langle h_{j},\gamma\rangle \ \ \text{ if } \gamma\in\Gamma_{j}.
                  \end{cases}
\end{equation}

To prove (\ref{equ non-sim thm 3 c}), we only need to combine (\ref{equ non-sim thm 3 d}) and (\ref{equ non-sim thm 3 e}) and there are three cases to consider. Case 1: $\gamma\in\Gamma^{\sigma(j)}$ and $\gamma\in\Gamma_{j}$. Case 2. $\gamma\in\Gamma_{\sigma(j)}$ and $\gamma\in\Gamma^{j}$. Case 3: $\gamma\in\Gamma_{\sigma(j)}$ and $\gamma\in\Gamma_{j}$. The remaining part is completely similar to the last part in the proof of Lemma \ref{lem non simply lace key} and hence is omitted.

\bigskip
\par\noindent {\bf Acknowledgments.}
The first author would like to thank Prof. Henning Krause and Prof. Claus. M. Ringel for their consistent support when he was a postdoctor in Fakult\"{a}t f\"{u}r Mathematik, Universit\"{a}t Bielefeld, where a large part of this work was done. He is also very grateful to Prof. Jan Schr\"{o}er for interesting discussions. The second author would like to thank Prof. Jan Schr\"{o}er for valuable support.

\bibliographystyle{amsplain}

\begin{thebibliography}{30}
\bibitem{An} J. E. Anderson, A polytope calculus for semisimple groups, \textit{Duke Math. J.} 116 (2003), 567-588.

\bibitem{BK} P. Baumann and J. Kamnitzer, Preprojective algebras and MV polytopes, \textit{Represent. Theory} 16 (2012), 152-188.

\bibitem{CB} W. Crawley-Boevey, On the exeptional fibres of Kleinian singularities, \textit{Amer. J. Math.} 122 (2000), 1027-1037.

\bibitem{CS} W. Crawley-Boevey and Jan Schr\"{o}er, Irreducible components of varieties of modules, \textit{J. reine angew. Math} 553 (2002), 201-220.

\bibitem{FRS} J. Fuchs, U. Ray and C. Schweigert, Some automorphisms of generalized Kac-Moody algebras, \textit{J. Algebra} 191 (1997), 518-540.

\bibitem{GLS2} C. Gei{\ss}, B. Leclerc and J. Schr\"{o}er, Kac-Moody groups and cluster algebras, \textit{Adv. Math.} 228 (2011), 329-433.

\bibitem{H} J. Hong, Mirkovi\'{c}-Vilonen cycles and polytopes for a symmetric pair, \textit{Represent. Theory} 13 (2009), 19-32.

\bibitem{J} Y. Jiang, Parametrizations of canonical bases and irreducible components of nilpotent varieties, \textit{Int. Math. Res. Not. IMRN} 2014, 3263-3278.

\bibitem{Kac} V. G. Kac, Infinite dimensional Lie algebras, 3rd edition, Cambridge University Press, 1990.

\bibitem{Kam1} J. Kamnitzer, The crystal structure on the set of Mirkovi\'{c}-Vilonen polytopes, \textit{Adv. Math.} 215 (2007), 66-93.

\bibitem{Kam2} J. Kamnitzer, Mirkovi\'{c}-Vilonen cycles and polytopes, \textit{Ann. of Math.} 171 (2010), 245-294.

\bibitem{Ka1} M. Kashiwara, On crystal bases of the $q$-anlogue of universal enveloping algebras, \textit{Duke Math. J.} 63 (1991), 456-516.

\bibitem{Ka2} M. Kashiwara, The crystal base and Littelmann's refined Demazure character formula, \textit{Duke Math. J.} 71 (1993), 839-858.

\bibitem{KS} M. Kashiwara and Y. Saito, Geometric construction of crystal bases, \textit{Duke Math. J.} 89 (1997), 9-36.

\bibitem{Lu1} G. Lusztig, Canonical bases arising from quantized enveloping algebras, \textit{J. Amer. Math. Soc.} 3 (1990), 447-498.

\bibitem{MV} I. Mirkovi\'{c} and K. Vilonen, Langlands duality and representations of algebraic groups over commutative rings, \textit{Ann. Math.} 166 (2007), 95-143.

\bibitem{R} C. M. Ringel, The preprojective algebra of a quiver, Algebras and modules II (Geiranger, 1966), CMS Conf. Proc. 24, AMS (1998), 467-480.

\bibitem{NS} S. Naito and D. Sagaki, A modification of the Anderson-Mirkovi\'{c} conjecture for Mirkovi\'{c}-Vilonen polytopes in types $B$ and $C$, \textit{J. Algebra} 320 (2008), 387-416.

\bibitem{Sai2} Y. Saito, Mirkovi\'{c}-Vilonen polytopes and a quiver construction of crystal basis of type $A$, \textit{Int. Math. Res. Not. IMRN} 2012, 3877-3928.

\end{thebibliography}

\end{document}